\documentclass[reqno]{amsart}
\usepackage{graphicx} % Required for inserting images
\usepackage{mathptmx}     
\usepackage{amsmath,amsfonts,amsthm,amssymb, txfonts}
\usepackage{mathtools}
\usepackage{comment} % provides \begin{comment}
\usepackage[dvipsnames]{xcolor}
\usepackage[pagebackref=false,colorlinks]{hyperref}
\hypersetup{citecolor={rgb,256:red,170;green,50;blue,170},
linkcolor={rgb,256:red,;green,60;blue,160}}
\usepackage{sseq}
\usepackage[T1]{fontenc}
\usepackage{calrsfs}
\usepackage{mathrsfs}
\DeclareMathAlphabet{\pazocal}{OMS}{zplm}{m}{n}

\usepackage{tikz-cd}
\usepackage{cleveref}
\usepackage[margin=1.5 in]{geometry} 

\usepackage[all]{xy}

\usepackage{todonotes}

\usepackage{calrsfs}

\newtheorem{theorem}{Theorem}[section]
\newtheorem{corollary}[theorem]{Corollary}

\newtheorem{lemma}[theorem]{Lemma} 

\newtheorem{proposition}[theorem]{Proposition}

\newtheorem*{theoremA}{Theorem A}
\newtheorem*{theoremB}{Theorem B}

\theoremstyle{definition}
\newtheorem{definition}[theorem]{Definition}

\newtheorem{remark}[theorem]{Remark}
\newtheorem{construction}[theorem]{Construction}
\newtheorem{example}[theorem]{Example}

\newcommand{\G}{\mathbb{G}}

\newcommand{\End}{\mathrm{End}}

\newcommand{\Gal}{\mathrm{Gal}}
\newcommand{\Pic}{\mathrm{Pic}}
\newcommand{\Picnalg}{\mathrm{Pic}_n^{alg}}

\newcommand{\DD}{\mathbb{D}}
\newcommand{\FF}{\mathbb{F}}
\newcommand{\GG}{\mathbb{G}}
\newcommand{\QQ}{\mathbb{Q}}
\renewcommand{\SS}{\mathbb{S}}
\newcommand{\WW}{\mathbb{W}}
\DeclareMathOperator{\Aut}{Aut}
\DeclareMathOperator{\colim}{colim}

\DeclareMathOperator{\Ext}{Ext}

\DeclareMathOperator{\Hom}{Hom}

\DeclareMathOperator{\Sp}{\pazocal{S}p}

\newcommand{\ZZ}{\mathbb{Z}}
\newcommand{\pic}{\mathfrak{pic}}
\newcommand{\Kn}{\mathrm{K(n)}}
\newcommand{\En}{\mathrm{E_n}}
\newcommand{\E}{\mathbf{E}}
\newcommand{\barE}{\overline{\E}}
\newcommand{\barR}{\overline{R}}
\newcommand{\EnG}[1]{\mathrm{E}^{h{#1}}_n} 
\newcommand{\SKn}{\mathrm{S}^0_\Kn}
\newcommand{\Gn}{\GG_n}

\newcommand{\gr}{\mathrm{gr}}

\newcommand{\Einf}{\pazocal{E}_\infty}
\newcommand{\Ecomm}[1]{\pazocal{E}_{#1}}

\newcommand{\tH}{\widehat{H}}

\newcommand{\kappak}[1]{\kappa_{n,#1}} %descent filtration
\newcommand{\kappaNk}[2]{\kappa_{n,#2}^{#1}} % descent filtration for a closed subgroup
\newcommand{\kappaN}[1]{\kappa_{n}^{#1}} % kappa for a closed subgroup

\newcommand{\Gring}[1]{\ZZ_p[\![#1]\!]} %group ring over Z_p

\newcommand{\xpic}{x_{\varspade}} %class in the pic spectral sequence
\newcommand{\ypic}{y_{\varspade}}
\newcommand{\zpic}{z_{\varspade}}

\usepackage{soul}

\DeclareSymbolFont{extraup}{U}{zavm}{m}{n}
\DeclareMathSymbol{\varspade}{\mathalpha}{extraup}{85}

\title{Bounding the $K(p-1)$-local exotic Picard group at $p>3$}

\author[Bobkova]{Irina Bobkova}
\address{Department of Mathematics, Texas A\&M University, College Station, TX, 77843, USA}
\email{ibobkova@tamu.edu}

\author[Lachmann]{Andrea Lachmann}
\address{Bergische Universit\"at Wuppertal, Fachgruppe Mathematik und Informatik, Gau{\ss}stra{\ss}e 20, 42119 Wuppertal, Germany}
\email{lachmann@uni-wuppertal.de}

\author[Li]{Ang Li}
\address{Mathematics Department, University of California, Santa Cruz
1156 High Street, Santa Cruz, CA 95064}
\email{ali169@ucsc.edu}

\author[Lima]{Alicia Lima}
\address{Department of Mathematics, The University of Chicago, Eckhart Hall, 5734 S University Ave, Chicago, IL 60637}
\email{alima@uchicago.edu}

\author[Stojanoska]{Vesna Stojanoska}
\address{Department of Mathematics, University of Illinois, Urbana-Champaign, 273 Altgeld Hall
1409 W. Green Street, 
Urbana, IL 61801}
\email{vesna@illinois.edu}

\author[Zhang]{Adela YiYu Zhang}
\address{Department of Mathematical Sciences,
University of Copenhagen,
Universitetsparken 5, 2100 Copenhagen, Denmark}
\email{yz@math.ku.dk}

\date{}

\begin{document}
\maketitle
\begin{abstract}
   In this paper, we bound the descent filtration of the exotic Picard group $\kappa_n$, for a prime number $p>3$ and $n=p-1$. Our method involves a detailed comparison of the Picard spectral sequence, the homotopy fixed point spectral sequence, and an auxiliary $\beta$-inverted homotopy fixed point spectral sequence whose input is the Farrell-Tate cohomology of the Morava stabilizer group. Along the way, we deduce that the $\Kn$-local Adams-Novikov spectral sequence for the sphere has a horizontal vanishing line at $3n^2+1$ on the $E_{2n^2+2}$-page.

   The same analysis also allows us to express the exotic Picard group of $\Kn$-local modules over the homotopy fixed points spectrum $\mathrm{E}_n^{hN}$, where $N$ is the normalizer in $\Gn$ of a finite cyclic subgroup of order $p$, as a subquotient of a single continuous cohomology group $H^{2n+1}(N,\pi_{2n}\En)$. 
\end{abstract}

\setcounter{tocdepth}{1}
\tableofcontents

\section{Introduction}
A key objective of chromatic homotopy theory is to understand the $\Kn$-local category of spectra $\Sp_{\Kn}$, at various primes $p$ and heights $n$, as these are building blocks of spectra from the chromatic point of view. 
Furthermore, $\Kn$-local spectra are approachable via descent from their Lubin-Tate  homology (aka Morava $E$-theory). 
Specifically, there is a $\Kn$-local even-periodic $\Einf$-ring spectrum $\En$, whose $\pi_0$ is a complete local ring carrying a universal deformation of a height $n$ formal group law $\Gamma$ over a field $k$ of characteristic $p$ \cite{GoerssHopkins, Lurie-Elliptic}.

The automorphism group $\Gn$ of the pair $(\Gamma,k)$ 
acts on $\En$ through ring homomorphisms, and the homotopy fixed points of the action recover the $\Kn$-local sphere \cite{devinatzhopkins}. In fact, this relationship can be categorified to great effect, exhibiting the $\Kn$-local category as the homotopy fixed points of the category of $\Gn$-equivariant $\Kn$-local $\En$-modules \cite{mathew2016galois,mor}.

The invertible objects in $\Sp_{\Kn}$, in turn, can be thought of as the building blocks of the $\Kn$-local category. Equipped with the $\Kn$-local smash product, they form the Picard group
\[ \Pic_n = \Pic(\Sp_{\Kn})= \{X \in \Sp_{\Kn} |\; \exists Y\ \text{ such that } X \otimes Y \simeq \SKn\}/\sim, \] where the equivalence relation is $\Kn$-local homotopy equivalence.

The investigation into $\Pic_n$ was initiated by the groundbreaking work of Hopkins, first recounted by Strickland \cite{Strikland-p-adic}, and then expanded by Hopkins, Mahowald, and Sadofsky \cite{hopkins-mahowald-sadovsky}. These works observe the wealth of information contained in $\Pic_n$. Indeed, while the category of (unlocalized) spectra has Picard group $\ZZ$, consisting of only the integer-dimensional spheres, the structure of $\Pic_n$ encodes the existence of $p$-adic dimensional spheres \cite{hopkins-mahowald-sadovsky} and determinant spheres \cite{BBGS}, among other beasts.
A starting point in the study of invertible $\Kn$-local spectra is their characterization as such spectra $X$ for which $(E_n)_*X$ is invertible in the category of $\pi_*\En-\GG_n$ modules (i.e. modules over $\pi_*\En$ with compatible $\GG_n$ action; these are known as Morava modules \cite[Definition 3.37]{BB}). We let $\Picnalg$ denote the Picard group of Morava modules,
and we consider the comparison map
\begin{equation}\label{eq:comparePic}
    \varepsilon: \Pic_n \to \Picnalg
\end{equation}
which sends a spectrum $X$ to $(\En)_*X=\pi_*L_{K(n)}(\En\wedge X)$.
Since $\pi_0\En$ is a complete local ring, an invertible $\pi_*\En$-module is completely determined by whether it is concentrated in even or odd degrees. Thus, most of the information in the algebraic Picard group is encoded by twists of the $\Gn$-action on $\pi_*\En$.

The map in \eqref{eq:comparePic} is known to be an isomorphism when $2(p-1)>n^2+n$ \cite{Pstragowski}. In case $p-1>n^2+n$, this reflects the more subtle notion of algebraicity of the $\Kn$-local category \cite{Pstragowski-KnAlg}. For small primes, $\varepsilon$ is not injective, and its kernel $\kappa_n$ is the group of \emph{exotic} invertible $\Kn$-local spectra. The few results that are known about the non-trivial exotic Picard groups can be summarized as follows:
\begin{itemize}
    \item At height $n=1$, $\kappa_1$ is non-zero only if $p=2$, in which case it is $\ZZ/2$ \cite{hopkins-mahowald-sadovsky}; 
    \item At height $n=2$, $\kappa_2$ is non-trivial only in two cases:
    \begin{itemize}
        \item When $p=3$, $\kappa_2 \cong \ZZ/3 \times \ZZ/3$ \cite{ghmrpicard}, and
        \item When $p=2$, $\kappa_2 \cong (\ZZ/8)^2 \times (\ZZ/2)^3$ \cite{BBGHPS};
    \end{itemize}
    \item For all $p$ and $n=p-1$, $\kappa_n$ contains a non-trivial subgroup of order $p$ \cite[Theorem 14.6]{BGHS-duality};
    \item At the prime $2$, $\kappa_n$ contains a cyclic subgroup of order $n+2 -v_2(2n+3+(-1)^n) $, where $v_2(-)$ is the $2$-adic valuation \cite[Theorem 6.5]{HeardLiShi};
    
    \item For odd primes $p$, $\kappa_n$ is a product of cyclic $p$-groups, according to \cite[Theorem 4.4.1]{heard2014thesis};
    
    \item Conditional upon a homological conjecture, $\kappa_3 $ is trivial at the prime $5$ \cite[Theorem 3.32]{culver}.
   \end{itemize}

In this paper we will further explore $\kappa_n$ and related groups when $n=p-1$. This is a boundary case for complication: the Morava stabilizer group at $n=p-1$ has infinite cohomological dimension because it contains finite torsion subgroups, but this is minimally complicated as the order of any finite $p$-torsion subgroup is exactly $p$. 

The analysis of $\kappa_2$ at $p=2$ in \cite{BBGHPS} contains a compendium of almost all the known strategies for approaching $\kappa_n$, and in particular, it made clear that understanding two natural filtrations on $\kappa_n$ can be crucial for clarifying the structure of this group. In this paper we study the descent filtration on $\kappa_n$ (cf. \cite[Definition 3.28]{BBGHPS}, \Cref{def:filterKappaGk}). It is most naturally described as the filtration arising from the descent spectral sequence for the Picard spectrum of the $\Kn$-local category. We review this in more detail in  \Cref{sec:Pic}.

To approach this descent spectral sequence, we first dive into a close study of the $\Kn$-local Adams-Novikov spectral sequence for the sphere, i.e. the homotopy fixed point spectral sequence\footnote{Here, and everywhere in this paper, group cohomology of profinite groups is always taken continuously.}
\begin{equation}\label{eq:hfp}
    H^s(\Gn,\pi_t\En) \Rightarrow \pi_{t-s} \SKn.
\end{equation}
While the group cohomology $H^*(\Gn,\pi_*\En)$ is generally inaccessible, it is well-understood in high cohomological degrees, at least in the case $n=p-1$. Namely,  above the $p$-virtual cohomological dimension of $\Gn$, which is $n^2$, $H^s(\Gn,\pi_*\En)$ is isomorphic to the  Farrell-Tate cohomology $\tH^s(\Gn,\pi_*\En)$ explicitly computed by Symonds \cite{Symonds} for $n=p-1$, which is the height we focus on. One key computational feature of Farrell-Tate cohomology is that for profinite groups of finite virtual cohomological dimension, whose finite $p$-Sylow subgroups are cyclic, their Farrell-Tate cohomology reduces to the Farrell-Tate cohomology of the normalizers of cyclic subgroups of order $p$ \cite[Theorem 1.3]{Symonds}. This result is a strengthening of Henn's f-isomorphism theorem \cite[Theorem 1.4]{henncentralizers}.

In the case of $\Gn$ at height $n=p-1$, any non-trivial finite $p$-subgroup of $\Gn$ is isomorphic to the cyclic group $C_p$, and its normalizer $N$ is such that the quotient $N/C_p$ has virtual cohomological dimension $n$. In fact, $N/C_p$ is an extension of $\ZZ_p^n$ and a finite group of order prime to $p$, making its Farrell-Tate cohomology simple enough to compute. Then it is a result of Symonds \cite[Theorem 1.1]{Symonds} that there is an isomorphism 
\begin{equation}\label{eq:FarrellTateintro}
\tH^*(\Gn,\pi_*\En) \cong \tH^*(N,\pi_*\En) \cong \tH^*(F,\pi_*\En)\otimes  \Lambda_{\FF_p}(a_0,\dots, a_{n-1}),
\end{equation}
where $F$ is a maximal finite subgroup containing our $C_p$. The Tate cohomology of $F$ with coefficients in $\pi_*\En$ can be explicitly determined from a well-known computation due to Hopkins and Miller; see \cite{Nave}. The exterior generators arise from the cohomology of $N/C_p$. We review these results in \Cref{sec:Tate} below.

The Farrell-Tate cohomology in \eqref{eq:FarrellTateintro} is in fact the $E_2$-page of the $\beta$-localized homotopy fixed point spectral sequence \eqref{eq:hfp}, where $\beta$ is a cohomology class in $H^{2}(\GG,\pi_{2pn}\En)$ detecting the class $\beta_1\in \pi_{2pn-2}\SKn$. There is a similar $\beta$-inverted homotopy fixed point spectral sequence for a finite subgroup $F$ of $\GG$ as well as for the normalizer $N$. Their relationship and full computation is described in \Cref{cor: tate splits}. 

While the $\beta$-inverted spectral sequences converge to zero, they encode crucial information (in high enough cohomological degree) about the homotopy fixed point spectral sequence \eqref{eq:hfp}, as well as the analogous homotopy fixed point spectral sequence for the action of $N$ on $\En$. A detailed analysis of this information yields the following explicit horizontal vanishing line, which we prove in \Cref{sec:boundFiltration}.

\begin{theoremA}[\Cref{cor: vanishing line}]
   Let $n=p-1$ for the prime $p\geq 3$, and let $G$ be $N$ or $\Gn$. There is a horizontal vanishing line $s=2n^2+\mathrm{vcd}(G)+1$ on the $E_{2n^2+2}$-page of the homotopy fixed points spectral sequence 
   \[E^{s,t}_2= H^{s}(G,\pi_t\En)\Rightarrow\pi_{t-s}(\EnG{G}).\]
\end{theoremA}

    Of course, nilpotence technology ensures the existence of a horizontal vanishing line on some finite page of the homotopy fixed point spectral sequence for $\EnG{G}$ for any closed subgroup $G\subseteq \Gn$, cf. \cite[Lemma 5.11]{devinatzhopkins} and \cite[Corollary 2.3.10]{beaudry2022chromatic}. Nonetheless, there are few examples where the exact bound is known, and our result contributes another class of such examples.

    \begin{remark}
    When $p=3$, we recover the horizontal vanishing line (at $s=13$ on the  $E_{10}$-page) from \cite[Theorem 4.2]{ghmrpicard}.
    Furthermore, in op.cit, the authors demonstrated that the line is sharp by using computations from \cite{HennKaramanovMahowald} to find elements in $E^{12,96+72k}_{10}$ that are detected by the Adams-Novikov spectral sequence of the Smith-Toda complex $V(0)$, see the paragraph after \cite[Lemma 4.7]{ghmrpicard}. 
    Unfortunately, our analysis does not prove sharpness for $p> 3$, as we do not have fine enough information about the fate of the Tate cohomology classes that could interact with the $\beta$-torsion in the homotopy fixed point spectral sequence for the sphere \eqref{eq:hfp} or  a suitable generalized Smith-Toda complex.%
    \end{remark}

    \begin{remark}
         In comparison, if $p-1$ does not divide the height $n$, then there is a horizontal vanishing line $s=n^2+1$ on the $E_2$-page. This is because $H^*(\Gn, \pi_*\En) $ is isomorphic to the Galois fixed points of the cohomology of the small Morava stabilizer group $\mathbb{S}_n$, which has no $p$-torsion and thus its cohomological dimension is $n^2$; see for example (\cite[Theorem 3.2.1]{henncentralizers} and \cite[Proposition 1.13]{goerss2021comparing}).
    \end{remark}

The complete understanding of the homotopy fixed point spectral sequence \eqref{eq:hfp} in high enough degrees also yields useful information for the Picard group $\Pic_n$ of the $\Kn$-local category. Namely, we use the additive-to-Picard comparison of differentials from \cite{ms} (see \Cref{thm: comparison picard hfpss}) to deduce differentials in the Picard spectrum homotopy fixed point spectral sequence. Note that this is made possible in the profinite case via the recent descent results of Mor \cite{mor}. 

To wit, consider the Picard spectrum $\pic(\En)$ of the category of $\Kn$-local $\En$-modules. It has a natural $\Gn$-action, and for a closed subgroup $G$ of $\Gn$ such as $N$ when $n=p-1$ or all of $\Gn$, we have a spectral sequence
\begin{equation}
    H^s(G,\pi_t\pic(\En)) \Rightarrow \pi_{t-s}\left(\pic(\En)\right)^{hG}.
\end{equation}
Here, $\pi_0$ of the abutment is the Picard group of $\Kn$-local $\EnG{G}$-modules. Thus, the spectral sequence induces a filtration on this Picard group, and the subgroup of filtration $s\geq 2$ is the group $\kappaN{G}$ of exotic invertible $\Kn$-local $\EnG{G}$-modules; see \Cref{def:filterKappaGk}. This is in agreement with $\kappa_n = \kappaN{\Gn}$ being the kernel of the map $\varepsilon$ in \eqref{eq:comparePic}, since $H^0(G,\pi_0\pic(\En)) = \ZZ/2$ and $H^1(G,\pi_1(\pic(\En))) \cong H^1(G,(\pi_0\En)^\times)$ conspire to build the algebraic Picard group of $G$-equivariant $\pi_*\En$-modules.
 In particular, $\kappaN{G}$ itself inherits a filtration, called the descent filtration, and we deduce the following bound on its size. Note that part (1) in the following result comes from knowing that there is only one descent filtration jump in the case of the normalizer subgroup $N$.

\begin{theoremB}[\Cref{thm: main corollary}]\label{thm: main}
Let $p\geq 5$ be a prime and let $n=p-1$.
\begin{enumerate}
\item Let $N$ be the normalizer of $C_p\subset\GG_n$. The exotic Picard group of $\Kn$-local $\EnG{N}$-modules is a subquotient of $H^{2n+1}(N,\pi_t\En)$.  In particular, $\kappaN{N}$ is a finite group of simple $p$-torsion.
    \item The descent filtration on the exotic Picard group $\kappa_n$ has length at most $n^2$, and the associated graded is concentrated in  $\frac{n}{2}-1$ degrees.
\end{enumerate}
\end{theoremB}

\subsection*{Conventions}
Throughout this paper, we fix an odd prime number $p$ and a height $n$. To avoid clunky notation, we will omit the subscript $n$ from $\En$, denoting it by $\E = \En$, and then $\E_t$ will denote $\pi_t\En$. We will also omit the subscript from the Morava stabilizer group, thus $\GG$ denotes $\Gn$, and the small stabilizer group will be $\SS$. Other objects which depend on $n$ might also have notation which does not explicitly include $n$, but since $n$ is fixed throughout, there is little chance of confusion.

We will work exclusively in a $\Kn$-local setting, thus all spectra are implicitly or explicitly $\Kn$-local, and tensor products (i.e.\ smash products) are implicitly $\Kn$-localized as needed.

We denote cyclic groups of order $m$ by $C_m$. The center $\ZZ_p^\times$ of $\GG$ contains a finite subgroup of the roots of unity, which we denote by $\mu_{p-1}$. All group cohomology of profinite groups is continuous.

\subsection*{Acknowledgements}
The influence of Agn\`es Beaudry's ideas on this project is hard to overestimate; it is our pleasure to thank her for her unending generosity and enthusiasm for our work.
We also thank Paul Goerss, Hans-Werner Henn, and Peter Symonds for their support and their willingness to share their time and ideas on several aspects of this project.  Additionally, Stojanoska thanks Itamar Mor for explaining his descent result from \cite{mor} in the case of general closed subgroups of $\Gn$. 

During the writing of this article,
Lachmann is supported by the German Research Foundation DFG through the Research Training Group GRK 2240. Lima is supported by the GFSD and NSA Graduate Fellowship.
Zhang is supported by the European Union via the Marie Curie Postdoctoral Fellowship and the Danish National
Research Foundation through the Copenhagen Centre for Geometry and Topology (DNRF151). 
This research has additionally been supported by the National Science Foundation grants DMS-2239362, DMS-2220741, and DMS-2304797.
The authors would like to thank the Hausdorff Institute for Mathematics in Bonn for their hospitality during the Women in Topology IV workshop. 

\section{Preliminaries from chromatic homotopy theory}
We begin by recalling some standard notions and notation from chromatic homotopy theory. 
Fix a formal group law $\Gamma$ of height $n$ over $\FF_{p^n}$ that is already defined over $\FF_p$, for example the Honda formal group law  with $[p]$-series $[p](x) = x^{p^n}$. The Morava $K$-theory spectrum $\Kn$ is a 2-periodic complex-oriented cohomology theory with a formal group law $\Gamma$, and coefficients $\Kn_*=\FF_{p^n}[u^{\pm 1}]$, where $u$ is in degree $-2$.

The Morava stabilizer group $\G=\mathrm{Aut}(\Gamma,\FF_{p^n})$ is the group of automorphisms of the pair $(\Gamma,\FF_{p^n} )$, and the small Morava stabilizer group $\SS=\mathrm{Aut}(\Gamma/\FF_{p^n})$ is the group of automorphisms of $\Gamma$ over $\FF_{p^n}$. Since $\Gamma$ is defined over $\FF_p$, the Galois group $\Gal(\FF_{p^n}/\FF_p)$ acts on $\SS$ and there is a decomposition $$\G \cong \SS \rtimes \Gal(\FF_{p^n}/\FF_p).$$

Denote by $\E=\En$ the Morava (or Lubin-Tate) $E$-theory $\E(\Gamma, \FF_{p^n})$. Let $\E_0\cong W(\FF_{p^n})[\![u_1,\ldots,u_{n-1}]\!]$ be the ring that classifies deformations of $\Gamma$ \cite{LubinTate}. The coefficient ring of $\E$  is a Laurent polynomial ring on the ring $\E_0$
\[\E_*\cong W(\FF_{p^n})[\![u_1,\ldots,u_{n-1}]\!][u^{\pm 1}].\]
The formal group law of $\E$ is a universal deformation of $\Gamma$.
The Morava stabilizer group $\GG$ acts continuously on $\E_*$, and the Goerss-Hopkins-Miller theorem upgrades this action to an $\Einf$-ring action on $\E$ \cite{GoerssHopkins}. The homotopy fixed points  $\E^{h\G}$ with respect to this action recover the $\Kn$-local sphere $L_{\Kn}S^0$ \cite{devinatzhopkins}.

\subsection{Subgroups of the Morava stabilizer groups}%
\label{definition:subgroups}
In this subsection we let $n=p-1$, and introduce the subgroups $N$ and $F$ of the Morava stabilizer groups $\GG$, following \cite[Section 3.6]{henn}. The reader is referred to \cite{henn} for proofs and details on this rather brief summary.

Let $\Gamma$ be the Honda formal group law of height $n$. The endomorphism ring of $\Gamma$ can be described as a non-commutative algebra over the Witt vectors of $\FF_{p^n}$ in one generator $S$ satisfying $S^n = p$, i.e.\ 
\[ \End_{\FF_{p^n}}(\Gamma) \cong \WW_{\FF_{p^n}} \langle S \rangle /(S^n-p) \, . \]
It is the ring of integers of the division algebra $\mathbb D$ over $\QQ_p$ of dimension $n^2$ and Hasse invariant $1/n$. 
The Morava stabilizer group $\SS$ is the group of \textit{automorphisms} of $\Gamma$, i.e. it has the presentation $\SS \cong \End_{\FF_{p^n}}(\Gamma)^{\times}$. 

Choose a primitive $(p^n-1)$-st root of unity $\omega$ in $\mathbb{F}_{p^n}^{\times} \subseteq \SS$.
Denote by $X$ the element $\omega^{\frac{p-1}{2}}S$ in $\SS$, so $X^n = -p$. By \cite[Lemma 19]{henn}, the field $\QQ_p(X)$ is a subalgebra of $\DD$ isomorphic to $\QQ_p(\zeta_p)$ for some primitive $p$-th root of unity $\zeta_p$ in $\DD$, and this isomorphism restricts to $$\ZZ_p[X]/(X^n+p) \cong \ZZ_p[\zeta_p].$$ The element $\zeta_p$ is an algebraic integer and also a unit in $\DD$, thus it is an element of 
$\SS$. We let $C_p=\langle \zeta_p \rangle$; as a subgroup of $\GG$  it is unique up to conjugacy. 
We denote by $N=N_{\GG}(C_p)$ the normalizer of the elementary abelian subgroup $C_p$ in $\GG$.

Denote by $\tau$ the element \[ \tau=\omega^{\frac{p^n-1}{(p-1)^2}} \in \SS \, . \]
The two elements $X$ and $\tau^nX^2$ generate a subgroup of $\GG$ denoted by $H$. This group is isomorphic to $C_{2n} \times C_{n/2}$.

The elements $X$, $\zeta_p$, and $\tau$ generate a finite subgroup of $N$ of order $pn^3$ which we denote by $F$. The group $F$ is a maximal finite subgroup of $\GG$.

\begin{proposition}[{\cite[Proposition 20]{henn}}]\label{prop: extension of N} 
The subgroups $C_p, H, F,$ and $ N$  of $\GG$ are related as follows.
\begin{enumerate}
    \item There is a short exact sequence  
    \[1\rightarrow H\times C_p\times \ZZ_p^n   \rightarrow N\rightarrow \mathrm{Aut(C_p)}\rightarrow 1.\] 
    \item The subgroups $H$, $C_p$ and $\ZZ_p^n $ are invariant with respect to the action of $\mathrm{Aut(C_p)}$. The action on $C_p$ is the tautological action, while $\ZZ_p^n $ is isomorphic to the direct sum of the distinct one-dimensional $\ZZ_p$-representations of $\mathrm{Aut(C_p)}$.
    \item There is a short exact sequence
\[ 1 \to H \times C_p \to F \to \Aut(C_p) \to 1 \, . \]
\end{enumerate}
  \end{proposition}
\begin{remark}
     It follows that the quotient $N/C_p$ is a group of cohomological dimension $n$ at the prime $p$, as it is an extension of $\ZZ_p^n$ and finite subgroups of order prime to $p$.
\end{remark}

\subsection{Homotopy fixed points spectral sequences}
We begin by briefly summarizing some results of Devinatz-Hopkins \cite{devinatzhopkins} that relate the $\Kn$-local $\E$-based Adams-Novikov spectral sequence with the homotopy fixed points spectral sequence.
\begin{construction}\label{EASS}
    For any spectrum $X$, and an $ \Ecomm{1} $  (i.e. associative) ring spectrum $R$ we can form a cosimplicial object 
    \begin{center}
    \begin{tikzcd}
   R\otimes X\ar[r,shift left]\ar[r, shift right,swap]&R\otimes R\otimes X \ar[r,shift left=1.5]\ar{r}\ar[r, shift right=1.5,swap]& R\otimes R \otimes R\otimes X\ar[r,shift left=2]\ar[r,shift left=0.7]\ar[r, shift right=0.7]\ar[r, shift right=2]&\cdots
    \end{tikzcd}
\end{center}
in $\Sp$, 
which is obtained by smashing the Amitsur complex of the unit map $S^0\rightarrow R$ with $X$.
The $R$-based Adams-Novikov spectral sequence for $X $ can be obtained as the Bousfield-Kan spectral sequence associated to the totalization of this cosimplicial spectrum.

Alternatively, one constructs a ``filtered'' object giving the same spectral sequence after the $E_2$-page. Namely, denote by $\barR$ the fiber of the unit map $\SKn\to R$; then we have the canonical Adams-Novikov $R$-resolution
\begin{equation}\label{eq:Ebar-res}
 T_{\bullet}^R(X) \to X,
\end{equation}
where $T_m^R(X) = \barR^{\otimes m+1} \otimes X$, 
as in \cite[Definition 2.2.10]{ravenelbook}. To avoid any potential for confusion with the notion of resolution from \Cref{def: resolution of spectra}, we will refer to \eqref{eq:Ebar-res} as the Adams-Novikov \emph{tower} for $X$.
\end{construction}

The classical Adams-Novikov tower is the above based on the complex cobordism spectrum, or $p$-locally, on the Brown-Peterson spectrum $BP$. Throughout this paper we will be working $\Kn$-locally, in which case all the terms in \Cref{EASS} should be re-localized after tensoring.

After $\Kn$-local localization, the $E_2$-page of the $BP$-based Adams-Novikov spectral sequence is usefully identified with the $E_2$-page of the $\E$-based Adams-Novikov spectral sequence, which in turn becomes group cohomology of the stabilizer group. In fact, that is an identification on the level of spectral sequences, per the following result.

\begin{theorem}[\cite{MillerRavenel,devinatzhopkins}]\label{thm: EASS}
    For $G$ any closed subgroup of $\GG$,  the spectral sequence associated to the $\Kn$-localized Adams-Novikov tower of the homotopy fixed points spectrum $\E^{hG}$ is strongly convergent with signature
    \begin{equation}\label{eq:hfpss-S0}
        E_2^{s,t}(G,\E)=H^s(G,\E_t)\Rightarrow \pi_{t-s}(\E^{hG}).
    \end{equation}
   Furthermore, if $X$ is a dualizable object in the $\Kn$-local category, then the spectral sequence obtained by mapping $X$ into the $\Kn$-local Adams-Novikov tower of $\E^{hG}$ is strongly convergent with signature
    \begin{equation}\label{hfpss}
        E_2^{s,t}(G,\E\otimes X)=H^s(G,\E_t(X))\Rightarrow \pi_{t-s}(\E^{hG}\otimes X).
    \end{equation}
\end{theorem}

While very little can be said about the homotopy fixed point spectral sequence \eqref{eq:hfpss-S0} in general, if the subgroup $G$ contains the subgroup $\mu_{p-1}$ of $(p-1)$st roots of unity, which are central in $\GG$, then we have the following well-known sparseness result on the $E_2$-page. We record it here for convenience, as it will be used to reduce the size of some exotic Picard groups below in \Cref{sec:boundFiltration}.

\begin{proposition}\label{prop:sparse}
   Let $G$ be any closed subgroup of $\G$ containing the central subgroup $\mu_{p-1}$ of roots of unity. Then for any $s$, $H^s(G,\E_t)=0$ unless $t$ is divisible by $2(p-1)$.
\end{proposition}
\begin{proof} The argument here is the same as the argument in the standard sparsity result for the Adams-Novikov Spectral Sequence. For completion, we will just include the argument as given in Heard's thesis \cite[Proposition 4.2.1]{heard2014thesis}.

Given our assumption on G, we have the Lyndon-Hochschild-Serre spectral sequence
\begin{equation*}
         H^i(G/\mu_{p-1}, H^j(\mu_{p-1}, \E_t ))\Rightarrow H^{i+j}(G, \E_t).
    \end{equation*}    
For a given element $g \in \G$, one can find a description of the action  of $g_*$ on $\E_*$ in \cite{devinatz1995}. In the case of central elements in $\G$, such as elements in $\mu_{p-1}$, one can be very explicit: if $\zeta$ is a generator of $\mu_{p-1}$, then
\[\zeta_*u^k=\zeta^{k} u^k \; \; \textrm{and} \;\;\; \zeta u_{i}=u_i.\]
One should note that since the order of $\mu_{p-1}$ is coprime to $p$, hence invertible in $\E_*$, the group $H^j(\mu_{p-1}, \E_t )$ is zero unless $j=0$. So that means we are only left to compute the group $H^0(\mu_{p-1}, \E_t )$, which from the action of $\mu_{p-1}$ given above, one can see that $H^0(\mu_{p-1}, \E_t )$ is nonzero only when $t$ is a multiple of $2(p-1)$, finishing the proof.
\end{proof}

\subsection{Homotopy fixed point spectral sequence for \texorpdfstring{$\E^{hF}$}{\E\^hF}}
\textbf{Suppose now that $p$ is odd, and $n=p-1$.}  The homotopy fixed points spectral sequence \eqref{eq:hfpss-S0},
in case $G$ is a finite subgroup of $\GG$, is explicitly well-understood, due to Hopkins and Miller, and first published in \cite[Section 2]{Nave}. 
The starting point is the following calculation of the $E_2$-page modulo the transfers in the case $G=C_p$.

\begin{proposition}[{Hopkins-Miller, cf. \cite[Theorem 2.1]{Nave}, \cite[Proposition 2.6]{HMS}}]
    There is an exact sequence
    \[\E_t\xrightarrow{tr}H^s(C_p,\E_t)\to \FF_{p^n}[\alpha,\beta,\delta^{\pm}]/(\alpha^2)\to 0\]
    of bigraded groups, where the $(s,t)$-bidegrees are $|\alpha|=(1,2n)$, $|\beta|=(2,2pn)$, and $|\delta|=(0,2p)$.
\end{proposition}
When $F\subset \GG$ is a maximal finite subgroup containing $C_p$, we obtain a similar exact sequence. 
   
The following result follows from the explicit understanding of the homotopy fixed points spectral sequence of $\E^{hF}$, which was first recorded in \cite[Section 2]{Nave}, and the fact that the generators on its $E_2$-page are invariant under the action by the Galois group.
\begin{theorem}[Hopkins-Miller, cf. \cite{HMS}, Lemma 2.8]\label{thm: hfpss for E^hF}
Modulo the image of the transfer map $tr: \E_* \to H^*(F, \E_*)$, the $E_2$-page of the homotopy fixed point spectral sequence 
\begin{equation}\label{sseq: HFPSS for F}
    E_2^{s,t}(F,\E) = H^s(F, \E_t) \Longrightarrow \pi_{t-s} \E^{hF},
\end{equation} is given by
\[
E_2^{*,*}/(\mathrm{tr})=H^*(F, \E_*)/(\mathrm{tr}) \cong \FF_{p}[\alpha, \beta, \Delta^{\pm 1}]/(\alpha^2)
\]
 with $|\alpha|=(1,2n)$, $|\beta|=(2,2pn)$, and $|\Delta|=(0,2pn^2)$. Along the line $s=0$, classes are concentrated in degrees $t = t-s$ divisible by $2n$.

The differentials are generated multiplicatively by
\begin{equation}\label{eq:hfpss differentials}
    d_{2n+1}(\Delta)=d_{2p-1}(\Delta)=\alpha \beta^{n}\quad \quad \text{ and }\quad\quad
    d_{2n^2+1}(\Delta^n\alpha)=\beta^{n^2+1},
\end{equation}
up to units, with $E_{\infty}(F,\E)=E_{2n^2+2}(F,\E)$. The class $\Delta^p$ is a permanent cycle and a periodicity generator for $\pi_*\E^{hF}$.  
\end{theorem}

\begin{remark}\label{rem:beta1origin}
The differentials of this spectral sequence are deduced by a comparison with the classical Adams-Novikov spectral sequence, via the composition \[\Ext_{BP_*BP}(BP_*,BP_*)\to H^*(\GG,\E_*)\to H^*(F,\E_*).\] In particular, under this map  the element $\beta_1\in\pi_{2pn-2}(S^0)$ is detected by the permanent cycle $\beta$, cf. \cite{Nave} and \cite{Ravenel_Arf}.
\end{remark}

\subsection{Finite resolutions of \texorpdfstring{$\E^{hN}$}{\E\^hN} and \texorpdfstring{$\E^{h\GG}$}{\E\^hG}}
As can be glimpsed from \Cref{thm: EASS}, the $\Kn$-local category of spectra $\Sp_{\Kn}$ is largely controlled by the continuous cohomology of the  Morava stabilizer group $\G$. As a result, the homological properties of $\GG$ are reflected in homotopy. For one, the existence of a finite length Adams $\E$-resolution for $\SKn$ is closely connected to the existence of a finite length projective resolution of the trivial $\G$-module $\mathbb{Z}_p$. However, that can only happen in case $(p-1)$ does not divide $n$; see \cite[Theorem 4]{henn}. Otherwise the small stabilizer group $\SS$ has infinite virtual cohomological dimension at $p$.

In \cite{GHMR}, Goerss, Henn, Mahowald, and Rezk pioneered the study of finite resolutions of the $\Kn$-local sphere by spectra which are not $\E$-injective or flat over $\E$, but are nonetheless well understood. While \cite{GHMR} deals with the case of $p=3$ and $n=2$, \cite{henn} discusses similar resolutions at more general heights and primes. In particular, in op.cit., Henn constructs a resolution of the $\Kn$-local sphere and the related spectrum $\E^{hN}$ at height $n=p-1$ for arbitrary odd primes $p$. We review those resolutions here, as they will play an important role in proving \Cref{thm: tate splits} below.

Throughout this paper we will use the term ``resolution'' of spectra
which we define below.
\begin{definition}[{\cite[Section 3.3.1]{henn}}]\label{def: resolution of spectra} A sequence of spectra
\begin{equation}\label{eq:resn}
    * \rightarrow X \rightarrow X_{0}\rightarrow X_{1} \rightarrow \cdots 
\end{equation}
is a \emph{resolution} of $X$ if the composite of any two consecutive maps is null-homotopic, and any of the maps $X_i \rightarrow X_{i+1}$ for $i \geq 0$ can be written as $X_i \rightarrow C_{i} \rightarrow X_{i+1}$ such that each $C_{i-1} \rightarrow X_{i} \rightarrow C_{i}$ is a cofibration for every $i \geq 0$. Here, our convention is that $C_{-1}:=X$. We  say that the resolution is of \textit{length n} if $C_{n} \cong X_n$ and $X_i \cong *$ for $i> n$.
\end{definition}

\begin{remark}\label{rmk: towersseq}
    Note that this definition implies that the resolution can be refined to a tower of spectra 
\begin{equation}\label{eq:tower-of-fibrations}
\xymatrix@C=10pt@R=12.5pt{
 X \ar[dr] && F_0 \ar[ll] \ar[dr] &&F_1 \ar[ll] \ar[rd] &&F_2 \ar[rd] \ar[ll]  && \ldots \ar[ll]\\
&X_0 \ar@{-->}[ru] && \Sigma^{-1}X_1 \ar@{-->}[ru] && \Sigma^{-2}X_2 \ar@{-->}[ru] && \Sigma^{-3}X_3 \ar@{-->}[ru],
}
\end{equation}
in which each $F_{i}$ is the homotopy fiber of the map $F_{i-1} \to \Sigma^{-i}X_{i}$. Note
that $F_i=\Sigma^{-i-1}C_i$ for all $i\geq 0$. This tower gives rise to an associated \textit{resolution spectral sequence}
\begin{equation}\label{ss: resolutionSS}
E_1^{s,t} =\pi_{t} X_s \Longrightarrow \pi_{t-s}X,
\end{equation}
\end{remark}

For the remainder of this section, \textbf{let $n=p-1$ for an odd prime $p$.}
While a finite $\E$-based Adams resolution (i.e. tower) for $\SKn$ does not exist, in \cite[Section 3.6]{henn}, Henn constructs a finite resolution for it whose terms are wedge sums of suspensions of homotopy fixed points of $\E$ under the action of the finite subgroup $F$ of $\GG$, as well as certain retracts of $\E$.
Note that $\E^{hF}$ is well understood due to the \Cref{thm: hfpss for E^hF}. 

The first step towards a resolution of $\SKn$ is an algebraic resolution of $\ZZ_p$ as a $\Gring{\GG}$-module.
To construct this algebraic resolution, Henn considers the following short exact sequence of $\Gring{\GG}$-modules \cite[Prop. 17]{henn},
\begin{equation}\label{resolutionexact}
0 \leftarrow \ZZ_p \xleftarrow{\epsilon}  \ZZ_p\!\uparrow_{N}^{\GG} \; \xleftarrow{f} K \leftarrow 0\  
\end{equation}
where the map $\epsilon$ is the canonical augmentation map from the induced module to the trivial one, and $K$ is simply defined as the kernel of $\epsilon$. Thus, having appropriate resolutions of $K$ and $ \ZZ_p\!\uparrow_{N}^{\GG} $ would yield a resolution of $\ZZ_p$, by taking the total complex of the resulting double complex.

On the topological level, the exact sequence \eqref{resolutionexact} is realized by a cofibration
\begin{equation}\label{resolutionexacttop}
 \SKn\simeq \E^{h\GG} \xrightarrow{\epsilon}  \E^{hN} \rightarrow C.
\end{equation}
The finite algebraic resolution of $\ZZ_p\!\uparrow_{N}^{\GG}$ by permutation modules yields an analogous resolution of $\E^{hN}$. Similarly, the finite projective resolution for $K$ gives rise to a finite topological resolution of $C$.  Here we summarize these results in the following theorem; see \cite[Section 3.5, Proposition 17, and Section 3.6]{henn} for details.

\begin{theorem}[{\cite[Theorems 25 and 26]{henn}}]\label{thm: Henntopres}\label{thm: Henntopres for G}
Let $p>2$ be a prime number and let $n=p-1$. 
\begin{enumerate}
    \item There is a resolution of length $n$
\begin{equation}\label{eq:henn-resolution-N}
X_\bullet \colon * \rightarrow \E^{hN} \rightarrow X_{0} \rightarrow \cdots \rightarrow X_n \rightarrow * 
\end{equation}
The spectrum $X_0$ is equivalent to $\E^{hF}$, while for $r>0$ we have 
\[X_r \simeq \bigvee_{(i_1, \cdots, i_r)} \Sigma^{2p^2n (i_1+ \cdots + i_r)} \E^{hF}\]
where the wedge is taken over all sequences of integers $(i_1, \cdots, i_r)$ with $ 0\leq  i_1< i_2 \cdots <i_r \leq n-1$. 

\item There is a resolution of finite length $m>n$,
\begin{equation}\label{eq:henn-resolution-G}
    Z_{_\bullet}: *  \rightarrow \SKn \rightarrow Z_{0} \rightarrow \cdots \rightarrow Z_{m} \rightarrow * \
\end{equation}
The spectrum $Z_0$ is equivalent to $\E^{hF}$, for $r>n$ each $Z_r$ is a summand of a finite wedge of $\E$'s, while for $0< r \leq n$,
\[Z_r \simeq V_r \vee X_r\]
where the $X_r$'s are as in (1) and $V_r$ is a direct summand of a finite wedge sum of $\E$'s.
\end{enumerate}
\end{theorem}

 From a computational standpoint, the resolution  $ Z_\bullet$ of $\SKn$ may be inefficient, since it lacks an explicit description of all the $V_r$ terms. By construction, these terms are derived from an algebraic resolution of the $\Gring{\GG}$-module $K$, which is generally mysterious and encodes the difference between $H^*(\GG)$ and $H^*(N)$.  Nonetheless, the close relationship between the resolution $X_\bullet$ of $\E^{hN} $ and $Z_\bullet$ of $\SKn$ was the main inspiration for us to first attempt to understand the Picard group of $\Kn$-local $\E^{hN}$-modules, as a step toward that of the $\Kn$-local category.
    
\begin{remark}
     The length $m$ of the resolution $Z_\bullet$ can be explicitly bounded. Namely, there is a resolution of length $n^2$ stemming from the fact that there exists a minimal-length algebraic resolution of the $\Gring{\GG}$-module $K$ \cite{scheiderer,symondsres}.
\end{remark}

\begin{remark}
Note that in the case $p=3$, the resolution of \Cref{thm: Henntopres}(2) is very different from the duality resolution of \cite{GHMR}.
 Nonetheless, just as in \cite{ghmrpicard}, this resolution can be key to understanding the exotic $\Kn$-local Picard group $\kappa_n$. 
\end{remark}

\begin{remark}\label{rem: ext gens}
    By construction, the terms in the resolution $\E^{hN}\to X_\bullet$ are indexed by a graded exterior algebra $\Lambda(a_0,\dots, a_{n-1})$, cf. \cite[Proposition 21]{henn}, such that $X_1$ can be thought of as $\bigvee_i a_i \E^{hF}$. More precisely, there is an equivalence 
    \[ \bigvee_{r=0}^{n} X_r \simeq \Lambda(a_0,\dots,a_{n-1})\otimes \E^{hF}. \]
    Compare with \Cref{rem:alphabetaconsistent} below. Note, however, that since the resolution $X_\bullet$ is not constructed to have any multiplicative properties, this is only an additive equivalence which nonetheless underlies the multiplicative structure in the Farrell-Tate cohomology of $N$.
    \end{remark}

%%%%%%%%%%%%%%%%%%%%%%%%%%%%%%%%%%%%%%%%%%%%%%%%%%%%%%%%%%%%%%%%%%%%%%%%%

\section{Farrell-Tate cohomology with coefficients in \texorpdfstring{$\E_*$}{E\_*}}\label{sec:Tate}
A first obstruction to understanding the $\Kn$-local $\E$-based Adams-Novikov, or homotopy fixed point spectral sequence \eqref{eq:hfpss-S0}, is computing its $E_2$-page, namely the continuous group cohomology of a closed subbroup of the Morava stabilizer group with coefficients in $\E_*$. While a complete determination is out of reach, we can get partial information by passing to the Farrell-Tate cohomology, which is controlled by the maximal finite subgroups.\footnote{For the precise statement, see \cite[Theorem 1.3]{Symonds} or \cite[Theorem 7.3]{symondsres}.} For $G=\GG$ or $N$ at the heights $n=p-1$, this results in a full and explicit computation in high degrees (more precisely, above the virtual cohomological dimension). The topological companion of this comparison is discussed in \Cref{sec:betainvert}. In this section we will recall the background and relevant computations at $n=p-1$. 

First, let us briefly recall some properties of the Farrell-Tate cohomology of profinite groups of finite virtual cohomological dimension; for details, the reader can consult \cite{scheiderer,symondsres}. Scheiderer's approach closely follows the classical case for (infinite) discrete groups; details can be found in \cite[X.2-3]{brown}. 

Suppose $G$ is a profinite group, satisfying the $FP_\infty$ finiteness criterion over $\ZZ_p$; in other words, there exists a resolution of the trivial module $\ZZ_p$ by finitely generated projective $\Gring{G}$-modules. We also assume that the virtual cohomological dimension $d$ of $G$ is finite. It is well-established that the Morava stabilizer groups $\G$ (and thus any of its closed subgroups), at any height and prime, satisfy these conditions.

In this situation, the trivial $G$-module $\ZZ_p$ has a complete resolution $F_\bullet$ by finitely generated projectives, and for a discrete or compact $\Gring{G}$-module $M$, the Farrell-Tate cohomology $\tH^*(G,M)$ can be defined as the cohomology of the complex of continuous $G$-homomorphisms from $F_\bullet $ to $M$ \cite[Definition 4.1]{scheiderer}. When $G$ is finite, this is exactly the standard definition of Tate cohomology, which in positive degrees agrees with ordinary cohomology. The analogue of that latter fact in the positive virtual dimension case fact is the following comparison result.

\begin{theorem}\cite{scheiderer,symondsres}\label{thm: tate vs cohom}
    Let $G$ be a profinite group with $\mathrm{vcd}(G)=d<\infty$. Suppose that $M$ is either a discrete or compact $\ZZ_p[\![G]\!]$-module. Then the  canonical map $$ H^s(G, M)\rightarrow\tH^s(G, M)$$ is onto for $s\geq d$ and an isomorphism for $s>d$.
\end{theorem}

\begin{example}
  The natural action of $\GG$ on $\E_*$ makes $\E_*$ a compact $\Gring{\GG}$-module. Restricting to $N$ makes $\E_*$ a compact $\Gring{N}$-module.
\end{example}

\begin{theorem} \cite[Theorem 1.1]{Symonds}\label{symonds}
    There is an isomorphism of bigraded algebras
\begin{equation}\label{eq:TateGN}
\tH^*(\GG, \E_*) \cong \tH^*(N,\E_*) \cong \FF_p[\alpha, \beta^{\pm 1}, \Delta^{\pm 1}]/(\alpha^2) \otimes_{\FF_p} \Lambda(a_0,\dots,a_{n-1}).
\end{equation}
The bidegrees of the generators are given in the following table 
\begin{center}
\begin{tabular}{ l | c | c }
& s & t\\ \hline
$\alpha$ & 1 & $2n$ \\ \hline
$\beta$ & 2 & $2pn$ \\\hline
$\Delta$ & 0 & $2pn^2$\\\hline
$a_i$ & $1$ & $2p^2ni$
\end{tabular}
\end{center}
in which having a bidegree $(s,t)$ corresponds to being an element of $\tH^s(G,\E_t)$.
\end{theorem}
\begin{remark}\label{rem:alphabetaconsistent}
The classes $\alpha, \beta, \Delta$ are chosen so that they map to their namesakes in the Tate cohomology of the finite subgroup $F$
\begin{equation}\label{eq:TateF}
\tH^*(F, E_*)  \cong \FF_p[\alpha, \beta^{\pm 1}, \Delta^{\pm 1}]/(\alpha^2).
\end{equation}
    Note that our choice of generators is \emph{different} from Symonds' in \cite[Proposition 2.3]{Symonds}, but is compatible with the Hopkins-Miller computation \Cref{thm: hfpss for E^hF}. We choose the exterior generators $a_i$ to be in one to one correspondence with each copy of $\Sigma^{2p^2ni}\E^{hF}$ in the term $X_1$ of the resolution of $\E^{hN}$ in \Cref{thm: Henntopres}. Our $a_i$'s can be obtained from Symonds' $x_i$'s via multiplication by powers of $\beta \Delta \in \tH^{2}(\GG,\E_{2p^2n}) $. More precisely, for $0\leq i \leq n-1$, we have $a_i = x_i\beta^i \Delta^{i}$.
\end{remark}

\begin{proposition}\label{prop: invert beta in group cohom}
Let $G $ be $N$ or  $\GG$. The natural map $H^*(G,\E_*) \to \tH^*(G,\E_*)$ can be identified with the $\beta$-inversion map
\[\varphi_G: H^*(G,\E_*) \to \beta^{-1}H^*(G,\E_*).\]
In particular, the $\beta$-inverted group cohomology $\beta^{-1}H^*(G,\E_*)$ is isomorphic to the Farrell--Tate cohomology $\tH^*(G,\E_*)$. Thus the map $\varphi_G$ is onto in cohomological dimensions $s \geq \mathrm{vcd}(G)$, and an isomorphism for $s > \mathrm{vcd}(G)$.
\end{proposition}
\begin{proof}
Given Theorem \ref{symonds}, it suffices to prove the case $G=N$. 
The statement follows by tracing through the proof of \cite[Theorem 1.1]{Symonds}. First, note that for the cyclic group $C_p$, the $\beta$-inversion map $\varphi_{C_p}$ identifies $\beta^{-1} H^*(C_p,\E_*) $ with $\tH^*(C_p,\E_*)$; this is classical and the interested reader can find details in loc.cit., for example. As in loc.cit., there is a spectral sequence
\begin{equation}\label{eq:tateSS-beta}
H^*(N/C_p, \tH^*(C_p, \E_*)) \Rightarrow \tH^*(N,\E_*), 
\end{equation}
whose $E_2$-term is now identified with $\beta^{-1}H^*(N/C_p, H^*(C_p,M))$, since $\beta$ is an $N$-invariant permanent cycle. Thus we identify \eqref{eq:tateSS-beta} with the $\beta$-inverted Lyndon-Hochschild-Serre spectral sequence, and the claim follows.
\end{proof}

%%%%%%%%%%%%%%%%%%%%%%%%%%%%%%%%%Old section 5 
\section{The \texorpdfstring{$\beta$}{beta}-inverted spectral sequences}\label{sec:betainvert}

The homotopy fixed points spectral sequences for $\E^{hN}$ and $\E^{h\GG}$ are difficult to understand fully or directly. However, we saw in \Cref{symonds} that the Farrell-Tate cohomology, unlike the ordinary continuous cohomology, is readily computable. By \Cref{prop: invert beta in group cohom}, the passage from ordinary to Farrell-Tate cohomology amounts to inverting the class $\beta$. Thus, while we lose information about $\beta$-torsion classes, the Farrell-Tate cohomology retains information about ordinary cohomology classes which are not killed by powers of $\beta$. 

By \Cref{rem:alphabetaconsistent} and \Cref{thm: hfpss for E^hF}, the class $\beta$ detects its namesake in $\pi_*\E^{hF}$, which is well-known to be (up to a unit) the Hurewicz image of the element $\beta_1\in\pi_{2pn-2}(S^0)$; see for example \cite{Ravenel_Arf}. Since $\beta_1$ is nilpotent, its inversion results in the zero spectrum, so in particular $\beta_1^{-1}\E^{hG}$ is contractible. Nonetheless, the $\beta$-inverted homotopy fixed point spectral sequence can be used to determine differentials in the non-$\beta$-inverted one, despite its convergence to zero. The strategy of using $\beta$-inverted or Tate spectral sequence to deduce information is well-established; see for example \cite{stojanoskaDuality}. Our argument is inspired by the analogous one at $p=3$ in \cite[Section 4.2]{ghmrpicard}, where the $\beta$-inverted homotopy fixed point spectral sequence and the resolution spectral sequence are played off against each other.

The exterior classes in the Farrell-Tate cohomology in \Cref{symonds} are closely related to the generators (i.e. shifted units of the summands) of the spectra $X_i$ in the resolution of $\E^{hN}$ in \Cref{thm: Henntopres}; see \Cref{rem: ext gens}. Understanding their behavior in the homotopy fixed point spectral sequence will be done by comparison to their behavior in the resolution spectral sequence \eqref{ss: resolutionSS}. We will review how such comparisons can be done in the first subsection, and then we pass to inverting $\beta$ in the next.

\subsection{Combining the two towers for \texorpdfstring{$\E^{hN}$}{E\^hN} and \texorpdfstring{$\E^{h\GG}$}{E\^hG}}
Combining Henn's finite resolutions of  $\E^{hN}$ and $\E^{h\GG}$ from \Cref{thm: Henntopres} with their Adams-Novikov resolutions as in \Cref{EASS}, results in squares of spectral sequences, which will help us relate their differentials. This was done in \cite{beaudry2022chromatic}, see the construction on page 401 of op.cit. Since we will further invert $\beta$ in this square of spectral sequences, we review the construction.

\begin{construction}\label{constr: bires for E^hG}\label{rem: exterior bires} 
    Let $G$ be one of $N $ or $\GG$, and let $F_\bullet(G)$ denote the tower
    $$  F_\bullet(G) \to \E^{hG} $$
    constructed as in \Cref{rmk: towersseq} from the corresponding finite resolution in \Cref{thm: Henntopres}. For each $F_r(G)$, let $  F_{r,\bullet}(G) \to F_r(G)$ denote its canonical $\E$-based $\Kn$-local Adams-Novikov tower, as in \eqref{eq:Ebar-res}. Naturality of the Adams-Novikov tower makes $F_{\bullet,\bullet}(G)$ into what could be called a double filtration of spectra.  

    This double filtration gives rise to a square of spectral sequences, by filtering in the two different directions, following the procedure 
    in \cite[Section 3]{miller}. We can first take the $E_1$-terms in the finite resolution direction. For specificity, note these are the $X_i$'s (resp. $Z_i$'s) in the resolution from \Cref{thm: Henntopres}, and the corresponding $d_1$-differentials are the maps in these resolutions. Naturality of Adams-Novikov towers now implies that we get maps between the Adams-Novikov towers of these $E_1$-terms. Thus the Adams-Novikov differentials will also commute with the finite resolution $d_1$-differentials. Altogether, taking first the resolution direction $E_1$-page, and then the Adams-Novikov $E_2$-page, yields
    \begin{equation}\label{eq:startingterms}
    \Lambda(a_0,\dots,a_{n-1})\otimes H^*(F,\E_*)
    \end{equation}
    in the case of $G=N$, due to \Cref{rem: ext gens} and \Cref{thm: EASS}. In the case of $G=\GG$, there are additional summands coming from the spectra $V_r$ in part (2) of \Cref{thm: Henntopres for G}.
    
    The expression \eqref{eq:startingterms} is the beginning of two spectral sequences: assembling it in the Adams-Novikov filtration gives rise to the homotopy fixed point spectral sequence, whereas assembling it in the finite resolution direction gives an algebraic resolution spectral sequence. The abutments of these spectral sequences in turn are the starting pages of two more spectral sequences. 

    Altogether, in the case of $G=N$, this becomes
\begin{equation}\label{fig: sqaure of sseq for E^hN}
    \begin{tikzcd}
        H^*(N,\E_*)\ar[r,-{Implies},double,"\mathrm{HFP}"]&\pi_*(\E^{hN})\\
         \Lambda(a_0,\ldots,a_{n-1})\otimes H^*(F,\E_*)\ar[u,-{Implies},double,""]\ar[r,-{Implies},double,"\mathrm{HFP}"]& \Lambda(a_0,\ldots,a_{n-1})\otimes \pi_*(\E^{hF}).\ar[u,-{Implies},double,"",swap]
    \end{tikzcd}    
\end{equation}
The horizontal arrows denote the homotopy fixed points spectral sequences arising from the Adams-Novikov tower direction, while the vertical ones arise from the finite resolution in \Cref{constr: bires for E^hG}.

There is a similar square of spectral sequences for $\GG$ instead of $N$, but its starting corner is slightly more mysterious, due to the $V_r$ summands in part (2) of \Cref{thm: Henntopres}. Indeed, we have the following square
\begin{equation}\label{fig: sqaure of sseq for E^hG}
    \begin{tikzcd}
        H^*(\GG,\E_*)\ar[r,-{Implies},double,"\mathrm{HFP}"]&\pi_*(\E^{h\GG})\\
         H^*(F,\E_*)\otimes\Lambda(a_0,\ldots,a_{n-1})\ar[u,-{Implies},double,""]\ar[r,-{Implies},double,"\mathrm{HFP}"]&\pi_*(\E^{hF})\otimes\Lambda(a_0,\ldots,a_{n-1})\ar[u,-{Implies},double,"",swap]\\[-22pt]
         \oplus H^0(\GG,\bigoplus_{r=0}^m \Hom(Q_r,\E_*))&\oplus H^0(\GG,\bigoplus_{r=0}^m \Hom(Q_r,\E_*))
    \end{tikzcd},
\end{equation}
where each $Q_r$ is a finitely generated projective $\Gring{\GG}$-module related to the $V_r$ from \Cref{thm: Henntopres for G}. Our $Q_r$ is denoted by $Q_r'$ in \cite[Theorem 26]{henn}.
\end{construction}

\begin{remark}\label{rem: hfpss splits}
    In both \eqref{fig: sqaure of sseq for E^hN} and \eqref{fig: sqaure of sseq for E^hG}, the exterior algebra consists of permanent cycles for the bottom horizontal spectral sequence almost tautologically, due to the naturality of the Adams-Novikov differentials.
\end{remark}

\subsection{The \texorpdfstring{$\beta$}{beta}-inverted homotopy fixed point spectral sequences}\label{subsec: invert beta}
Now we turn to studying what happens after inverting $\beta$ in the two squares of spectral sequences \eqref{fig: sqaure of sseq for E^hN} and \eqref{fig: sqaure of sseq for E^hG}. While the $\beta$-inverted homotopy fixed point spectral sequence for the finite group $F$ has a solid footing as a Tate spectral sequence \cite{GreenleesMay}, we explain a construction of the others before working with them. 

As discussed above, the element $\beta_1 \in \pi_{2pn - 2} S^0$ is detected by a cohomology class \[\beta \in E_2^{2,2pn}(\GG,\E)\cong H^2 (\GG,\E_{2pn}).\] 
Let $\tilde\beta$ be a lift of $\beta$ to the $E_1$-page of the $\Kn$-local Adams-Novikov spectral sequence constructed from the $\Kn$-localized tower $T_\bullet(S^0) = T_\bullet^{\E}(S^0) \to S^0$ in \eqref{eq:Ebar-res}.  Since $\tilde\beta$ is a permanent cycle, it defines a map $\tilde\beta :S^{2pn} \to T_2(S^0)= \barE^{\otimes 3}$. 

For any $X $ then, the ``inclusion'' $\barE\to S^0$ allows us to extend $\tilde\beta$ to a map of Adams-Novikov towers
\[ \tilde\beta: T_\bullet (X) \to T_{\bullet+2}(X),\]
where we have suppressed the shift of internal degree from the notation.  The $\beta$-localized Adams-Novikov tower of $X$ is the colimit
\[\beta^{-1}T_\bullet(X) = \colim_{\tilde\beta}T_\bullet (X).   \]
It is a $\ZZ$-indexed diagram of spectra, which is natural in $X$, and which gives rise to a four-quadrant spectral sequence whose $E_2$-page is the $\beta$-inverted $E_2$-page arising from $T_\bullet(X)$. Furthermore, the localization map $T_\bullet(X) \to \beta^{-1}T_\bullet(X)$ gives rise to a map of spectral sequences with corresponding multiplicative properties.

While the convergence properties of whole plane spectral sequences are generally tricky (see \cite[Section 8]{boardmanConverges}), the diagram $\beta^{-1}T_\bullet(X)$ gives rise to a conditionally convergent spectral sequence since $T_\bullet(X)$ does. In our case, some power of the cohomology class $\beta$ is a target of an Adams-Novikov differential in the spectral sequence for the sphere. This implies that the $\beta$-localized spectral sequence for any $X$ will collapse to zero at a finite stage, and strong convergence follows from \cite[Theorem 8.10]{boardmanConverges}.

Now consider the ``double complex'' $F_{\bullet,\bullet}(G)$ from \Cref{constr: bires for E^hG}. It can also be viewed as a diagram of Adams-Novikov towers
\[ T_\bullet(F_d(G)) \to \dots \to T_\bullet(F_1(G))\to T_\bullet(F_0(G)), \]
and hence it gives rise to a diagram of $\beta$-inverted towers
\[ \beta^{-1} T_\bullet(F_d(G)) \to \dots \to \beta^{-1} T_\bullet(F_1(G))\to \beta^{-1} T_\bullet(F_0(G)).\]
This diagram, in turn, gives rise to a square of spectral sequences as in \eqref{fig: sqaure of sseq for E^hN} and \eqref{fig: sqaure of sseq for E^hG} above, and we record the corresponding result as follows.
 
\begin{theorem}\label{thm: beta-inverted square}
    For $G=N$ and $\GG$, there exists a commutative square of strongly convergent spectral sequences 
\begin{equation}\label{eq:diagram}
    \begin{tikzcd}
       \beta^{-1}H^*(G,\E_*)\cong \widehat H^*(G,\E_*)\ar[r,-{Implies},double,"\beta^{-1}\mathrm{HFP}"]&\pi_*(\beta_1^{-1}\E^{hG})=0\\
        \widehat H^*(F,E_*)\otimes\Lambda(a_0,\ldots,a_{n-1})\ar[u,-{Implies},double,"\beta^{-1}\mathrm{Alg}"]\ar[r,-{Implies},double,"\beta^{-1}\mathrm{HFP}"]&\pi_*(\beta_1^{-1}\E^{hF})\otimes\Lambda(a_0,\ldots,a_{n-1})=0\ar[u,-{Implies},double,"",swap]
    \end{tikzcd}.
\end{equation}
\end{theorem} 
\begin{proof}
    The only thing that remains is to identify the terms. First, note that when $G=\GG$, the contributions from the projective modules $Q_r$ in \eqref{fig: sqaure of sseq for E^hG} are killed after $\beta$-inversion since $\beta$ is in positive cohomological dimension.
    Now apply \Cref{prop: invert beta in group cohom} to identify the $\beta$-inverted ordinary group cohomology with the Farrell-Tate cohomology.
\end{proof}

\begin{remark}\label{rem:bottom beta inverted}
In the $\beta$-inverted Adams-Novikov spectral sequence along the bottom of \eqref{eq:diagram}, the exterior algebra $\Lambda(a_0,\dots,a_{n-1})$ consists of permanent cycles by \Cref{rem: hfpss splits}, and thus the spectral sequence can be thought of as the $\Lambda(a_0,\dots,a_{n-1})$-tensored $\beta$-inverted $F$-homotopy fixed point spectral sequence. The latter is identified with the Tate spectral sequence for the action of $F$ on $\E$, and it appears in Heard's thesis \cite{heard2014thesis} and in his preprint \cite{drewTate}. We record it here as a corollary of the Hopkins-Miller computation from \Cref{thm: hfpss for E^hF}. 
\end{remark}

\begin{proposition}\label{prop: tate for E^hF}
    The $\beta$-inverted homotopy fixed point spectral sequence for $\E^{hF}$ takes the form
    \begin{equation}\label{sseq: tate for F}
        \beta^{-1}E_2^{s,t}(F,\E)=\tH^s(F;\E_t)\cong \FF_p[\alpha,\beta^{\pm 1},\Delta^{\pm 1}]/(\alpha^2)\Rightarrow\pi_{t-s}(\beta_1^{-1}\E^{hF})=0.
    \end{equation}
    Its differentials are multiplicatively generated by the formulas in \eqref{eq:hfpss differentials}, and the spectral sequence collapses to zero on the $E_{2n^2+2}$-page.
\end{proposition}
\begin{proof}
     Since $\beta_1$ is detected by $\beta$ on the $E_2$-page, inverting $\beta_1$ in the homotopy fixed points spectral sequence  \[E^{s,t}_2(F,\E)=H^s(F,\E_t)\Rightarrow\pi_{t-s}(\E^{hF})\] is the same as inverting  the element $\beta$ on the $E_2$-page, which is  $\FF_p[\alpha,\beta,\Delta^{\pm 1}]/(\alpha^2)$ modulo transfer elements. 
     It follows from \Cref{thm: hfpss for E^hF} and the fact that transfer elements are $\beta$-torsion that the $\beta$-inverted homotopy fixed points spectral sequence has $E_2$-page $\FF_p[\alpha,\beta^{\pm 1},\Delta^{\pm 1}]/(\alpha^2)$. 
     The formulas in \eqref{eq:hfpss differentials} along with the Leibniz rule imply that the spectral sequence collapses to 0 on the $E_{2n^2+2}$-page, since the unit is hit by a differential; namely $d_{2n^2+1} (\alpha \beta^{-1-n^2}\Delta^n) = 1$.
\end{proof}

\begin{remark}\label{rem: all differentials}
    Note that $\alpha$, $\beta$, and $\Delta^p$ are permanent cycles, and we can give explicit formulas for all differentials in the spectral sequence \eqref{sseq: tate for F}. Namely, we have that $d_{2n+1}(\beta^m\Delta^k) = k \alpha \beta^{m+n} \Delta^{k-1}$, which is non-zero if and only if $k$ is not divisible by $p$, while all the $\alpha$-multiples are $(2n+1)$-cycles since $\alpha^2=0$. Thus, the $E_{2n+2}$ page is generated by classes of form $\beta^m\Delta^{pk}$ and $\alpha \beta^m\Delta^{n+pl}$, for some integers $m,k,l$. Then $d_{2n^2+1}(\alpha\beta^m \Delta^{n+pk} ) = \beta^{m+n^2+1} \Delta^{pk}$ wipes away all the classes on $E_{2n^2+1}$.
\end{remark}

Next, we investigate the spectral sequence $\beta^{-1}\mathrm{Alg}$ in square (\ref{eq:diagram}).

\begin{proposition}\label{prop:a_i permanent}
    The left vertical spectral sequence $\beta^{-1}\mathrm{Alg}$ in diagram \eqref{eq:diagram} has no non-trivial differentials.
\end{proposition}
\begin{proof}
   This follows from \Cref{symonds}. There is a clear bijection between the $E_2$-page and the $E_\infty$-page, and any non-zero differentials would contradict \Cref{symonds}.  
\end{proof}

As a corollary of \Cref{rem:bottom beta inverted} and \Cref{prop:a_i permanent}, we conclude that all higher differentials in the $\beta$-inverted spectral sequence
\[\beta^{-1}E^{s,t}_2(G,\E)= \beta^{-1} H^*(G, \E_*)=\tH^s(G,\E_t)\Rightarrow\pi_{t-s}(\beta_1^{-1}\E^{hG})\]
i.e. the top horizontal spectral sequence in \eqref{eq:diagram}, 
come from the Tate spectral sequence from \eqref{prop: tate for E^hF}, i.e. the bottom horizontal spectral sequence in \eqref{eq:diagram}. We record this conclusion as the following result.

\begin{corollary}\label{cor: tate splits}\label{thm: tate splits}
    Let $G$ be one of $N$ or $\GG$. The $\beta$-inverted homotopy fixed points spectral sequence 
    \[\beta^{-1}E^{s,t}_2(G,\E)=\tH^s(G,\E_t)\Rightarrow\pi_{t-s}(\beta_1^{-1}\E^{hG})=0\] 
    splits as a direct sum of shifts of the Tate spectral sequence (\ref{sseq: tate for F}) for $\beta_1^{-1}\E^{hF}$ indexed over the monomial basis of the exterior algebra $\Lambda(a_0,\ldots,a_{n-1})$. 

In particular, the spectral sequence collapses to zero on the $E_{2n^2+2}$ page.
\end{corollary}

\section{Differentials detected by the \texorpdfstring{$\beta$}{beta}-inverted homotopy fixed point spectral sequence}
Let $n=p-1$, and let $G$ denote either $\GG$ or $N$. In this section, we use \Cref{thm: tate splits} to deduce 
differentials in the homotopy fixed points spectral sequence $E_*^{*,*}(G,\E)$ \eqref{eq:hfpss-S0} above the virtual cohomological dimension of $G$. This will lead to a horizontal vanishing line on the $E_{2n^2+2}$-page in \Cref{cor: vanishing line}. In \Cref{prop: degconsideration in tate}, we give a more detailed analysis of some classes that are relevant for the computation of the exotic Picard groups in \Cref{sec:boundFiltration}.

To do this, we consider the map of homotopy fixed point spectral sequences
  \begin{equation}\label{eq:diagram-varphi}
    \begin{tikzcd}[column sep=large]
       E^{*,*}_2(G,\E)= H^*(G,\E_*)\ar[r,-{Implies},double,"\mathrm{HFP}"]\ar[d, "\varphi"]&\pi_*(\E^{hG})\ar[d]\\
          \beta^{-1}E^{*,*}_2(G,\E)= \beta^{-1}H^*(G,\E_*)\cong\tH^*(G,\E_*)\ar[r,-{Implies},double,"\beta^{-1}\mathrm{HFP}"]&\pi_*(\beta_1^{-1}\E^{hG}).
    \end{tikzcd}
\end{equation}

The map $\varphi$ is an isomorphism on the $E_2$-page in cohomological dimensions $s>d=\mathrm{vcd}(G)$. To control the isomorphism range in further pages of these spectral sequences, we will need the following auxiliary lemma.

\begin{lemma}\label{lem:comparezerodiff}
    Suppose $\varphi_r:E_r^{s,t}\to \tilde{E}_r^{s,t}$ is a map of spectral sequences such that $\varphi_r$ is onto for $s\geq M_0$ and an isomorphism for $s\geq M_1 \geq M_0$. Then $\varphi_{r+1}$ is onto for $s\geq \max(M_0, M_1-r)$ and an isomorphism for $s\geq \max(M_1,M_0+r)$.
\end{lemma}
\begin{proof}
    We start by noting that our conditions imply that the map on boundaries $\varphi_r^B: B_r^{s,t} \to \tilde{B}_r^{s,t} $ and the map on cycles $\varphi_r^Z: Z_r^{s,t}\to \tilde{Z}_r^{s,t}$ are injective for $s\geq M_1$.
    
    We apply the snake lemma to the diagram
    \[\xymatrix{
    0 \ar[r] & Z_r^{s,*} \ar[r]\ar[d]_{\varphi^Z_r} & E_r^{s,*} \ar[r]\ar[d]^{\varphi_r}  & B_r^{s+r,*} \ar[r]\ar[d]^{\varphi^B_{r}} & 0 \\
    0 \ar[r] & \tilde{Z}_r^{s,*} \ar[r] & \tilde{E}_r^{s,*} \ar[r]  & \tilde{B}_r^{s+r,*} \ar[r] & 0
    }\]
    to conclude that $\varphi_{r}^B$ (in degree $s$) is onto for $s\geq M_0+r$. Furthermore, $\varphi^Z_r$ is onto provided that $\varphi_r$ is onto and $\varphi^B_r$ is injective, which is satisfied when $s\geq \max (M_0, M_1-r)$.
    
    Now consider the diagram 
    \[\xymatrix{
    0 \ar[r] & B_r^{s,*} \ar[r]\ar[d]_{\varphi^B_r} & Z_r^{s,*} \ar[r]\ar[d]^{\varphi^Z_r}  & E_{r+1}^{s,*} \ar[r]\ar[d]^{\varphi_{r+1}} & 0 \\
    0 \ar[r] & \tilde{B}_r^{s,*} \ar[r] & \tilde{Z}_r^{s,*} \ar[r]  & \tilde{E}_{r+1}^{s,*} \ar[r] & 0.
    }\]
    Another application of the snake lemma yields that $\varphi_{r+1}$ is onto whenever $\varphi_r^Z$ is, which we saw is satisfied when $s\geq \max (M_0, M_1-r)$. For $\varphi_{r+1}$ to be injective, it suffices that $\varphi_r^Z $ is injective (so $s\geq M_1$) and $\varphi_r^B$  is onto (so $s\geq M_0+r$).
\end{proof}

Combining this result with the computation of differentials in the $\beta$-inverted homotopy fixed point spectral sequence in \Cref{thm: tate splits}, we are now able to prove the following result about an explicit vanishing line of the homotopy fixed point spectral sequence. While we will not need this result in the analysis of the $\Kn$-local Picard groups, we record it here as it is of independent interest.

\begin{theorem}\label{cor: vanishing line}
  Let $n=p-1$ for the prime $p\geq 3$, and let $G$ be $N$ or $\GG$. There is a horizontal vanishing line $s=2n^2+\mathrm{vcd}(G)+1$ on the $E_{2n^2+2}$-page of the homotopy fixed points spectral sequence 
   \[E^{s,t}_2(G,\E)= H^{s}(G,\E_t)\Rightarrow\pi_{t-s}(\E^{hG}).\]
   In other words, $E_r^{s,t}(G,\E)=0$ for $s\geq 2n^2+\mathrm{vcd}(G)+1$,  all $t$, and $r\geq 2n^2+2$.
\end{theorem}

\begin{remark}
    In particular, in the case $G=\GG$, it follows from \Cref{thm: EASS} that the $\Kn$-local Adams-Novikov spectral sequence for the sphere has a horizontal vanishing line $s=3n^2+1$ at the $E_{2n^2+2}$-page.
\end{remark}

\begin{proof}
For brevity, let $E_r^{s,t} = E_r^{s,t}(G,\E)$ and $\tilde{E}_r^{s,t} = \beta^{-1}E_r^{s,t}(G,\E)$. We have a map of spectral sequences $\varphi_r: E_r^{s,t}\to\tilde{E}_r^{s,t}$, and we know that $\varphi_2$ is onto for $s\geq d = \mathrm{vcd}(G)$ and an isomorphism for $s> d$. It follows from an induction on $r$ using \Cref{lem:comparezerodiff} that $\varphi_r$ is onto for $s\geq d$ and an isomorphism for $s\geq d+r-1$. Since $\tilde{E}_r^{s,*}$ is zero when $r\geq 2n^2+2$ according to \Cref{thm: tate splits}, we conclude that $E_r^{s,*}$ is zero when $r\geq 2n^2+2$ and $s\geq d + 2n^2+1$ as claimed.
\end{proof}

Now we turn to a finer analysis of the classes on the vertical line $t-s=-1$ in the homotopy fixed point spectral sequence $E_*^{*,*}(G,\E)$. These groups provide an upper bound for the filtration quotients of the exotic Picard group of $\Kn$-local $\E^{hG}$-modules by \Cref{sec:descentfiltrations} below.

First we record some elementary facts that account for the supply of classes in $\tH^{t+1}(G,\E_{t})$ above the virtual cohomological dimension. For the rest of this section, we let $p\geq 5$, so that $n^2>2n+1$.
\begin{proposition} \label{prop: degconsideration in tate}
    Suppose that $G $ is $N$ or $\GG$, and $p\geq 5$.
\begin{enumerate}
    \item 
 For $s>vcd(G)$, the groups $H^s(G,\E_{t})\cong \tH^s(G,\E_{t})$ are zero unless $t=2n\epsilon+2pnl$ for some $\epsilon\in\{0,1\}$ and $l\in \mathbb Z$.
 
    \item In the $\beta$-inverted homotopy fixed points spectral sequence, let $x$ be a class in $\beta^{-1}E^{t+1,t}_2(G,\E) \cong \tH^{t+1}(G,\E_{t})$ with $n^2 \leq t\leq 4pn$. If $x$ survives to the $E_{2n+2}\cong E_{2n^2+1}$-page, it cannot be the target of a $d_{2n^2+1}$-differential.
    \end{enumerate}
\end{proposition}

\begin{proof} 
    (1) This follows from \Cref{symonds} and degree considerations. Note that we are not claiming much here: an arbitrary element in $\tH^*(G,\E_*)$ has the form $\alpha^\epsilon \beta^m \Delta^{k} a_0^{\epsilon_0}\dots a_{n-1}^{\epsilon_{n-1}}$ and has topological degree $t = 2n\epsilon + 2pnl$, where $l =m+nk+p\sum_{i=0}^{n-1} i\epsilon_i$.

    (2) If $n^2\leq t =2n\epsilon+2pnl\leq 4pn$, then $l=1$ and $\epsilon=0$ or 1. 
    If $\epsilon=1$, then $x$ has the form $\alpha\beta^m\Delta^{k}a_0^{\epsilon_0}\cdots a_{n-1}^{\epsilon_{n-1}}$, which can not be the target of a $d_{2n^2+1}$-differential by \Cref{rem: all differentials} and \Cref{cor: tate splits}.

    If $\epsilon=0$,  then again by \Cref{rem: all differentials}, $x$ has the form $\beta^m\Delta^{kp}a_0^{\epsilon_0}\cdots a_{n-1}^{\epsilon_{n-1}}$, where the variables $m, k\in\mathbb{Z},$ and $\epsilon_i\in\{0,1\}$ for $i=0,\ldots,n-1$ satisfy
\begin{equation}\label{eq: sdegree}
    2m+\sum_{i=0}^{n-1}  \epsilon_i=2pn+1, \text{ and }
\end{equation}
\begin{equation}\label{eq: tdegree}
    m+pnk+p\sum_{i=0}^{n-1} i\epsilon_i = 1.
\end{equation}
From \eqref{eq: tdegree}, we deduce that $m-1$ is divisible by $p$. Setting $m-1=ph$ with $h\in\mathbb{Z}$ and plugging into \eqref{eq: sdegree}, we obtain 
that $\sum_{i=0}^{n-1} \epsilon_i$ equals $2p(n-h)-1$. But each $\epsilon_i$ is either 0 or 1, implying that $0\leq 2p(n-h)-1\leq n  = p-1$, which is impossible.
\end{proof}

%%%%%%%%%%%%%%%%%%%%%%%%%%%%%%%%%%%%%%%%%%%%%%%%%%%%%%%%%%%%%%%%

\section{\texorpdfstring{$\Kn$}{K(n)}-local Picard groups}\label{sec:Pic} 

In this section we let $n\geq 1$ be an arbitrary height again, and introduce the objects of main interest in this paper, namely the various $\Kn$-local Picard groups. Recall that in wide generality, the Picard group of a (small enough) symmetric monoidal category $\mathcal{C}$ is the group of isomorphism classes of invertible objects in $\mathcal{C}$, equipped with the monoidal tensor operation. In chromatic homotopy theory, the Picard group of the $\Kn$-local stable homotopy category is usually denoted $\Pic_n$, and often called the Hopkins' Picard group honoring the fact that Mike Hopkins first observed how rich its structure can be.

One of the original tools for studying $\Pic_n$ is the fundamental exact sequence \cite{hopkins-mahowald-sadovsky}
\begin{equation}\label{eq:kappadef-orig}
    0 \to \kappa_n \to \Pic_n \xrightarrow{\varepsilon} \Picnalg \cong H^1(\G, \E_*^{\times}),
\end{equation}
determined by the map $\varepsilon$ that sends an invertible $\Kn$-local spectrum $X$ to its ($\Kn$-local) $\E$-homology, which is a graded invertible $\E_*$-module with a compatible $\G$-action. Since $\E_0$ is a complete local ring and $\E_* = \E_0[u^{\pm 1}]$, invertible $\E_*$-modules (without the $\G$-action) are determined by whether they are concentrated in even or odd degrees. In other words, the Picard group $\Pic (\E_*)$ is $\ZZ/2$. Denoting by $\Pic_n^0$ the kernel of the natural map $\Pic_n \to \Pic(\E_*)\cong \ZZ/2$, the sequence \eqref{eq:kappadef-orig} is refined to the following form
\begin{equation}\label{eq:kappadef}
    0 \to \kappa_n \to \Pic_n^0 \xrightarrow{\varepsilon} H^1(\G, (\E_0)^\times),
\end{equation}
which is often easier to work with since $ \E_0$ is not a graded ring.

The isomorphism $ \Picnalg \cong H^1(\G, (\E_*)^\times)$  \cite[Proposition 8.4]{hopkins-mahowald-sadovsky} gives cohomological description of this algebraic Picard group. While complete calculations of $\Picnalg$ are few and far between (for a computation of $\Pic_2^{alg}$ at $p=3$ see \cite{Karamanov}, and \cite{Lader} for $p>3$), understanding $\Picnalg$ may be best suited as a problem in arithmetic geometry \cite{GrossHopkins}, and homotopy theorists tend to focus their attention on the complementary information contained in $\kappa_n$. This is what we will do here as well.

The group $\kappa_n$ is simply defined as the kernel of $\varepsilon$, and is called the exotic $\Kn$-local Picard group. Its elements are those invertible $\Kn$-local spectra $X$ whose $\E$-homology is $\G$-equivariantly isomorphic to $\E_*$, thus they are exotic in the sense that they are not seen by the algebra of their $\E$-homology (i.e. by their Morava modules).

While our main interest is in the (exotic) Picard groups of the various $\Kn$-local categories of interest, for many purposes it is useful to think of them as the connected components of the respective Picard spaces, or $\pi_0$ of the Picard spectra. Given a presentable symmetric monoidal category $\mathcal{C}$, its Picard spectrum $\pic(\mathcal{C})$ is the connective spectrum obtained by delooping the $\infty$-groupoid of invertible objects of $\mathcal{C}$. See, for example, \cite{ms} or \cite{GepnerLawson} for more details.

As is usual, if $R$ is a commutative ring spectrum, we denote by $\Pic(R)$ and $\pic (R)$ the Picard group and spectrum of the category of $R$-modules. In fact, when $R$ is a $\Kn$-local ring spectrum, such as the $\Kn$-local sphere, the Lubin-Tate spectrum $\E$, or any homotopy fixed point spectrum $\E^{hG}$, \emph{we will denote by $\Pic(R)$ and $\pic(R)$ the Picard group and spectrum of the category of $\Kn$-local $R$-modules}.

\subsection{Descent for Picard groups}
The main appeal of studying $\Pic(\mathcal{C})$ as $\pi_0 \pic(\mathcal{C})$ is the amenability of the Picard spectrum to descent techniques. In \cite[Theorem 6.31]{GepnerLawson}, see also \cite[Section 3.3]{ms}, Picard spectrum descent was established for faithful finite Galois extensions. More recently, \cite{mor,lizhang} studied profinite Galois descent for Picard spectra in the $\Kn$-local setting. We will be mostly referencing \cite{mor}, as Mor's approach is more suitable for our intended applications.

Let $R$ be a ring spectrum. Then the homotopy groups of $\pic (R) $, the Picard spectrum of $R$-modules, are given by
\[\pi_t \pic (R) = \begin{cases}
    \Pic(R), \text{ for } t=0;\\
    \pi_0(R)^\times,  \text{ for } t=1;\\
    \pi_{t-1}  (R),  \text{ for } t>1.\\
\end{cases}\]
Suppose that $A\to B$ is a faithful $G$-Galois extension for a finite group $G$, which in particular implies that $A\simeq B^{hG}$. Then $\pic(A)$ is equivalent to the connective cover of $\pic(B)^{hG}$ \cite[Section 3.3]{ms}, and there is an associated homotopy fixed point spectral sequence, also called the Picard spectral sequence
\begin{equation}\label{sseq: picard}
    E_2^{s,t}(G,\pic (B)) = H^s(G, \pi_t \pic (B))  \Rightarrow \pi_{t-s}\pic(B)^{hG}. 
\end{equation} 
Restricting to $t-s\geq 0$, this spectral sequence computes $\pi_*(\pic(A))$, which yields $\Pic(A) \cong \pi_0(\pic(A))$. In particular, this gives a natural filtration of $\Pic(A) \cong \pi_0 \pic(B)^{hG}$, whose filtration quotients are $E_\infty^{s,s}(G,\pic (B))$ for $s\geq 0$. We will come back to this filtration below in \Cref{sec:descentfiltrations}.

 Furthermore, Mathew-Stojanoska obtained a general comparison tool (\cite[5.2.4]{ms}) to deduce differentials in the Picard spectral sequence \eqref{sseq: picard} (in a range) from those in the homotopy fixed points spectral sequence
 \begin{equation}\label{sseq: comparehfpss}
     E^{s,t}_2(G,B)= H^{s}(G,\pi_t(B))\Rightarrow\pi_{t-s}(B^{hG}).
 \end{equation}
The key observation is that there are equivalences
\begin{equation}\label{eq: truncate pic}
    \Sigma \tau_{[m,2m-1]}B\simeq \tau_{[m+1,2m]}\pic(B),
\end{equation}
for any $m\geq 2$ that are natural in the spectrum $B$. If $B$ is equipped with a $G$-action, then this equivalence is compatible with the $G$-action, so one can compare the differentials in \eqref{sseq: comparehfpss} with those in \eqref{sseq: picard} in a suitable range.

The quintessential Galois extension of chromatic homotopy theory, namely $S^0_{\Kn} \to \E$, has a profinite Galois group. The recent work \cite{mor} (see also \cite{lizhang}) establishes an analogous descent equivalence $\pic(S^0_{\Kn}) \simeq \tau_{\geq 0} \pic(\E)^{h\G}$, giving rise to an associated spectral sequence \eqref{sseq: picard} with $E_2$-page the continuous $\G$-cohomology of the homotopy groups of $\pic(\E)$. In fact, Mor's work also allows us to conclude that for any closed subgroup $G$ of $\G$, there is an equivalence 
\begin{equation}\label{eq:picishG}
    \pic(\E^{hG}) \simeq \tau_{\geq 0} \pic(\E)^{hG}.
\end{equation}
Thus we obtain an associated spectral sequence \eqref{sseq: picard}, whose $E_2$-page is the continuous cohomology of $G$; see also \cite[Corollary 3.3.14]{lizhang}. 
The natural equivalences \eqref{eq: truncate pic} then allow us to generalize Mathew-Stojanoska's comparison tool. 

\begin{theorem}[{\cite[Theorem A.IV]{mor}\cite[Theorem B, Corollary 3.3.14]{lizhang}}]\label{thm: comparison picard hfpss} Let $G$ be any closed subgroup of $\GG$, and consider the homotopy fixed point spectral sequences $E_r^{*,*}(G,\E)$ (as in \eqref{sseq: comparehfpss}) with differentials $d_{r,+}$, and $E_r^{*,*}(G,\pic (\E))$ (as in \eqref{sseq: picard}) with differentials $d_{r,\varspadesuit}$.

Let $x$ be an element in $E^{s,t}_2(G,\E)$, with $t\geq 2$, and let $x_\varspadesuit$ be the corresponding element in \[E_2^{s,t+1}(G,\pic\E) \cong E_2^{s,t}(G,\E).\] Given $2\leq r\leq t$, assume that $x$ survives to $E_r^{s,t}(G,\E)$, i.e. for all $q<r$, $d_{q,+}(x)=0$ and $x$ is not in the image of $d_{q,+}$. Then, 
$x_{\varspadesuit}$ survives to $E_{r}^{s,t+1}(G,\pic (\E))$ and $d_{r,\varspadesuit}(x_\varspadesuit) $  is identified with $d_{r,+}(x)$.
\end{theorem}

\subsection{The descent filtration on exotic Picard groups}\label{sec:descentfiltrations}
The Picard group of $\E^{hG}$ inherits a natural filtration $f_s\Pic(\E^{hG})$  from the descent spectral sequence, which in the case of $G=\G$ is closely related to \eqref{eq:kappadef-orig} and \eqref{eq:kappadef}, as well as the descent filtration on $\kappa_n$ from \cite[Construction 3.2]{ghmrpicard}, \cite[Section 3.3]{BBGHPS}, or \cite[Section 1.2, 1.3]{culver}. To be more precise, note that $\Pic(\E) \cong \Pic(\E_*) = \ZZ/2$, generated by the suspension shift. Thus, for any subgroup $G$ of $\G$, we have $H^0(G, \pi_0\pic(\E))\cong \ZZ/2$, and the bottom of the filtration is an exact sequence
\[ 0 \to f_1 \Pic(\E^{hG}) \to \Pic (\E^{hG}) \to E_{\infty}^{0,0}(G, \pic(\E))\cong E_{2}^{0,0}(G, \pic(\E))\cong\ZZ/2 \to 0,\]
where $f_1\Pic(\E^{hG})$ consists of those invertible $\E^{hG}$-modules $X$ for which $X \otimes_{\E^{hG}}\E$ is concentrated in even degrees. In particular, $f_1\Pic(\E^{h\G}) = \Pic_n^0$.

The next step in the filtration is the exact sequence
\[ 0 \to f_2\Pic(\E^{hG}) \to f_1 \Pic(\E^{hG}) \to E_\infty^{1,1}(G,\pic(\E))\subseteq H^1(G,\E_0^\times), \]
which is precisely \eqref{eq:kappadef} when $G=\G$. Thus, $\kappa_n = f_2 \Pic(\E^{h\G})$. In line with this example, we make the following definition.

\begin{definition}\label{def:filterKappaGk}
    Let $G$ be a closed subgroup of the Morava stabilizer group $\G$. The group $\kappaN{G}$ of exotic elements in the Picard group $\Pic(\E^{hG})$ is $f_2 \Pic(\E^{hG})$. Furthermore, the descent filtration on $\kappaN{G}$ is $\kappaNk{G}{s} = f_s \Pic(\E^{hG}) $ for $s\geq 2$.
\end{definition}    
    
\begin{remark}
    An equivalent description of $\kappaN{G}$ is as those $X\in \Pic(\E^{hG})$ such that $\pi_*(X \otimes_{\E^{hG}}\E)$ is $G$-equivariantly equivalent to $\E_*$. This group sits in an exact sequence 
    \begin{equation}\label{eq:kappaG}
        0 \to \kappaN{G}\to \Pic(\E^{hG})\to H^1(G, \E_*^\times).
    \end{equation}
\end{remark}  

    Note that by \Cref{def:filterKappaGk}, $\kappaNk{G}{s}$ comes with a map $\kappaNk{G}{s}\to E_s^{s,s}(G,\pic (\E))$, and we have a comparison of the target group with $E_s^{s,s-1}(G,\E)$ according to \Cref{thm: comparison picard hfpss}. In particular, the latter is a subquotient of $H^s(G,\E_{s-1})$. 
    
If the group $G$ contains the central roots of unity $\mu_{p-1}$, the sparsity result of \Cref{prop:sparse} implies a corresponding sparsity of the descent filtration on $\kappaN{G}$. We record it here for future reference. In this paper, we will use it in the proof of \Cref{thm: main corollary}.
\begin{lemma}\label{lem: filtration sparse}
    Assume the closed subgroup $G$ of $\G$ contains the central subgroup $\mu_{p-1} \in \G$. Then the associated graded of the descent filtration on $\kappaN{G}$ is concentrated in degrees congruent to $1$ modulo $2(p-1)$.
\end{lemma}

\begin{remark}
    The relationship of the descent filtration quotients $\kappaNk{G}{s}/\kappaNk{G}{s+1}$ to $H^s(G,\E_{s-1})$ can be used to describe the filtration without reference to the Picard spectral sequence \eqref{sseq: picard}. Instead, one studies the differential pattern of the Adams-Novikov spectral sequence \eqref{hfpss} for a representative exotic invertible $\Kn$-local spectrum. For details on this approach, the reader is referred to \cite[Section 3.3]{BBGHPS}.
\end{remark}

\begin{example}
    From \cite[Theorem 7.1.2]{ms}, we conclude that at the prime $2$ we have $\kappa_1^{C_2} = \ZZ/2$. From the computation in the proof of \cite[Theorem 8.1.3]{ms}, we deduce that at the prime $3$,  $\kappa_2^{F} = \ZZ/3$, for a maximal finite subgroup of $F$ of $\G$. Further, the proof of \cite[Theorem 4.1]{HMS} shows that if $G$ is a finite group containing $C_p$, then $\kappa_{p-1}^{G} = \ZZ/p$.
    
    At the prime $2$, larger groups appear; for example, for a maximal finite subgroup $F=G_{48}$, we conclude that $\kappa_2^{F} = \ZZ/8$ from \cite[Theorem 8.2.2]{ms}. 
    The group $\kappa_2^{C_4}$ at $p=2$ has order $4$, according to the proof of \cite[Proposition 7.4]{BBHS}.
\end{example}

\begin{remark}
Note that while $\kappaN{G}$ is not unrelated to the subgroup filtration on $\kappa_n$ from \cite[Section 3.1]{BBGHPS}, it is different. In particular, if $G_1 \subseteq G_2$ are nested closed subgroups of $\G$, then we have a map $\kappaN{G_2} \to \kappaN{G_1}$ making the diagram
\begin{equation*}
    \xymatrix{
    0 \ar[r] & \kappaN{G_2} \ar[r]\ar[d] &  \Pic(\E^{hG_2}) \ar[r]\ar[d] &  H^1(G_2, \E_*^\times)\ar[d] \\
    0 \ar[r] & \kappaN{G_1} \ar[r] &  \Pic(\E^{hG_1}) \ar[r] &  H^1(G_1, \E_*^\times) 
    }
\end{equation*}
commute, in which the middle and right-most maps are the natural ones. However, $\kappaN{G_2}$ need not be a subgroup of $\kappaN{G_1}$, and in fact none of the vertical maps need be inclusions. For a closed sugroup $G$ of $\GG$, the group $\kappa_n(G)$ of \cite{BBGHPS} is closely related to the kernel of $\kappa_n=\kappaN{\GG} \to \kappaN{G}$. 
\end{remark}

%%%%%%%%%%%%%%%%%%%%%%%%%%%%%%%%%%%%%%%%%%%%%%%%%
\section{Bounding the descent filtration}\label{sec:boundFiltration}

We are finally ready to apply the tools we have developed above and deduce certain differentials in the Picard spectral sequence \eqref{sseq: picard} with $B=\E$, \textbf{at height $n=p-1$}, and where $G$ is $\GG$ or $N$. As a result, we obtain a bound on the descent filtration of $\kappaN{N}$ and $\kappa_n$, as well as an explicit bound on the size of $\kappaN{N}$.

Using \Cref{thm: comparison picard hfpss} will allow us to compare an appropriate range of the Picard spectral sequence with the homotopy fixed point spectral sequence for the $G$ action on $\E$. We have partial but crucial information about the latter in \Cref{prop: degconsideration in tate}.

\begin{theorem}\label{thm: hfpsstopicard}
Suppose that $p\geq 5$, and let $G$ be $\GG$ or $N$. Let $\xpic$ be a class of bidegree $(t+1,t+1)$ in the $E_2$ page of the homotopy fixed point spectral sequence  
\begin{equation}\label{eq:sspic}
    E_2^{*,*}(G,\pic(\E)) \Rightarrow \pi_{*}\pic(\E)^{hG}.    
\end{equation}
Assume that $t\geq 1$, and that $\xpic$ is a non-trivial permanent cycle. Then
\begin{enumerate}
    \item when $G=N$, $t$ equals $2n$, while
    \item when $G=\GG$, $t$ is less than $n^2$.
    \end{enumerate}
\end{theorem}
\begin{proof}
First of all, note that for $m\geq 1$, we have an isomorphism $E_2^{s,m+1}(G,\pic(\E))\cong E_2^{s,m}(G,\E)$, so for each $y \in E_2^{s,m}(G,\E)$, denote by $\ypic$ the corresponding class in $E_2^{s,m+1}(G,\pic(\E))$. Since $\xpic$ has a companion class in $x \in E_2^{t+1,t}(G,\E)$, by the sparseness result of \Cref{prop:sparse}, we conclude that if $\xpic$ is to be non-trivial, $t$ must be at least $2n$.

(1) Suppose $G=N$ and  $\xpic \in E^{t+1,t+1}_2(G, \pic(\E))$ is a non-zero permanent cycle in cohomological degree above $\mathrm{vcd}(N) = n$. By looking at its companion $x\in E_2^{t+1,t}(G,\E)$, we deduce that $t$ has form $ 2n\epsilon +  2pnl$ by part (1) of \Cref{prop: degconsideration in tate}. We need to show that if $t>2n$, then $\xpic$ will be in the image of a differential. The formula $t = 2n\epsilon +  2pnl$ implies that if $t>2n$, then $t\geq 2pn = 2n^2 + 2n$. 

By \Cref{thm: comparison picard hfpss}, and the assumption that $\xpic$ is a permanent cycle, we conclude that $d_r(x) = 0$ for $r < 2n^2+2n $. Now \Cref{lem:comparezerodiff} implies that the image $\varphi(x)$ of $x$ in the $\beta$-inverted homotopy fixed point spectral sequence is a $d_r$-cycle for $r<2n^2+2n$. In light of \Cref{thm: tate splits}, this means $\varphi(x)$ is a permanent cycle. So, $\varphi(x)$ must be in the image of a differential, i.e. there exists either $\tilde{y}\in \beta^{-1}E_{2n+1}^{t-2n, t-2n}(G,\E)$ such that $d_{2n+1}(\tilde{y}) = \varphi(x)$, or  $\tilde{z}\in \beta^{-1}E_{2n^2+1}^{t-2n^2, t-2n^2}(G,\E)$ such that $d_{2n^2+1}(\tilde{z}) = \varphi(x)$. In either case, the cohomological degree of the class hitting $\varphi(x)$ is at least $2n$, thus invoking \Cref{lem:comparezerodiff} again gives that $x$ is the target of either a $d_{2n+1}$ or a $d_{2n^2+1}$ differential.

Suppose $d_{2n+1}(y) = x$ for some $y \in E_{2n+1}^{t-2n,t-2n}(G,\E)$. Then its topological degree $t-2n$ is at least $2n^2$, so by \Cref{thm: comparison picard hfpss}, the companion class $\ypic$ hits $\xpic$.  
If $\xpic$ is not in the image of $d_{2n+1}$, we conclude that there exists $z \in E_{2n^2+1}^{t-2n^2,t-2n^2}(G,\E)$ such that $d_{2n^2+1}(z) = x$. Part (2) of \Cref{prop: degconsideration in tate} now implies that $t > 4n^2$. Invoking \Cref{thm: comparison picard hfpss}, since the topological degree $t-2n^2$ of $z$ is at least $2n^2+1$, we conclude that $d_{2n^2+1}(\zpic) = \xpic$.

Altogether, we have that if $\xpic$ is a non-trivial permanent cycle in $E_*^{t+1,t+1}(N,\pic(\E))$, and $t\geq 1$, then $t = 2n$.

(2) Now we turn to the case $G=\GG$, which we argue in a similar fashion. We only need to show that if $t>n^2 = \mathrm{vcd}(G)$, then a permanent cycle $\xpic \in E_2^{t+1,t+1}(G,\pic(\E))$ must be hit by a differential. The companion class in $E_2^{t+1,t}(G,\E)$ will have topological degree $t = 2n\epsilon+ 2pnl$ by part (1) of \Cref{prop: degconsideration in tate}, implying that $l\geq 1$, i.e. $t \geq 2n^2+2n$.

The rest of the argument proceeds exactly as in the case of $G=N$. We argue that $x$ is a permanent cycle, by comparison to the $\beta$-inverted homotopy fixed point spectral sequence. If $\varphi(x)$ is in the image of a $d_{2n+1}$-differential, we import this differential to the Picard spectral sequence using \Cref{lem:comparezerodiff} and \Cref{thm: comparison picard hfpss}. If not, then it is in the image of a  $d_{2n^2+1}$-differential, but in this case we conclude $t > 4n^2 $ by \Cref{prop: degconsideration in tate}, which in turn allows us to import this differential to the Picard spectral sequence by another application of \Cref{thm: comparison picard hfpss}.
\end{proof}

With this result in hand, we are ready to read off the implications for the $\Kn$-local exotic Picard group. Recall that $\kappaN{G}$ is the subgroup of $\Pic(\E^{hG})$ of filtration 2 (and above) in the spectral sequence \eqref{eq:sspic}, and its descent filtration is the one inherited from this spectral sequence; cf. \Cref{def:filterKappaGk}. Thus, part (1) of \Cref{thm: hfpsstopicard} implies that $\kappaN{N}$ in fact equals $\kappaNk{N}{2n+1}$ and further, that $\kappaNk{N}{2n+2} = 0$. 

When $G$ is the whole group $\GG$, part (2) of \Cref{thm: hfpsstopicard} gives that $\kappak{n^2+1}= \kappaNk{\GG}{n^2+1} = 0$. While this much less precise, the sparsity result from \Cref{lem: filtration sparse} implies that the filtration quotients of $\kappa_n$ are concentrated in degrees congruent to $1$ modulo $2n$. 

\begin{corollary}\label{thm: main corollary}
    Suppose $p\geq 5$ and $n=p-1$. Let $N$ be the normalizer of $C_p\subset\GG$.
    \begin{enumerate}
        \item The exotic Picard group of $\Kn$-local $\E^{hN}$-modules is a subquotient of $H^{2n+1}(N,\E_{2n})$. In particular, $\kappaN{N}$ is a finite group of simple $p$-torsion. %and rank at most ?? 
        \item The descent filtration on the exotic Picard group $\kappa_n$ has length at most $n^2$, and its associated graded is concentrated in degrees congruent to $1$ modulo $2n$. More precisely,
        \[\gr_* \kappa_n \cong \bigoplus_{m=1}^{n/2-1} E_\infty^{2nm+1,2nm+1}(\GG,\pic(\E)),  \]
        and each $E_\infty^{2nm+1,2nm+1}(\GG,\pic(\E))$ is a subquotient of $H^{2nm+1}(\GG,\E_{2nm})$.
    \end{enumerate}
\end{corollary}

\begin{example}
    When $p=5$, part (2) gives that $\kappa_n$ itself is concentrated in a single degree and is a subquotient of $H^{9}(\GG,\E_8)$.
\end{example}

\begin{remark}
Unlike $H^{2n+1}(\GG,\E_{2n})$, the cohomology group $H^{2n+1}(N,\E_{2n})$ is not particularly mysterious. A combinatorial description of an $\FF_p$-basis of $H^{2n-1}(N,\E_{2n})\cong \tH^{2n-1}(N,\E_{2n})$ can be obtained as follows. By \Cref{symonds}, a generator with internal degree $2n$ has to be of the form  
$x = \alpha \beta^m \Delta^k a_0^{\epsilon_0}\cdots a_{n-1}^{\epsilon_{n-1}} $, where $k\in\mathbb Z$, $\epsilon_i \in\{0,1\}$, and $m = -nk - p\sum_{i=0}^{n-1} i\epsilon_i $. 
Hence we want to find all tuples $(k,\epsilon_0,\ldots,\epsilon_{n-1})$ such that the cohomological degree of $x$ is
    \begin{equation}\label{eq: basis for 2p-1}
        1+2m+\sum_{i=1}^{n-1}\epsilon_i=2n+1. 
    \end{equation}
    It follows that there is an even number of $\epsilon_i$'s that are 1, say those indexed by $0\leq i_1<\cdots<i_{2l}\leq n-1$. Then \eqref{eq: basis for 2p-1} can be rewritten as
    
    \begin{equation}\label{eq: comboforbasis}
        n \big( 1 + k + \sum_{j=1}^{2l}i_j \big) = l - \sum_{j=1}^{2l}i_j. 
    \end{equation}
    To find the solutions of this equation, as $l$ ranges from $0$ to $n/2$, we choose all $2l$-tuples $0\leq i_1 < \cdots < i_{2l} \leq n-1$, such that $l - \sum_{j=1}^{2l}i_j$ is divisible by $n$. (When $l=0$, all the $i_j$'s are zero.) Then $k$ is uniquely determined. So, the dimension of $H^{2n-1}(N,\E_{2n})$ as an $\FF_p$-vector space is given by the number of such choices of $l$'s and corresponding $i_j$'s. Unfortunately this combinatorial problem does not appear to have an explicit solution.
    
    A search of the first few values in the On-line Encyclopedia of Integer Sequences yields a match of the dimension of $H^{2n+1}(N,\E_{2n})$ with the number of periodic sequences of period $n$ made of an even number of $0$s and an even number of $1$s \cite[Table III, column $C_n$]{combinatorics}.
\end{remark}

\begin{example}
    When $p=5$, we get that $H^{2n+1}(N,\E_{2n}) $ is 4-dimensional, at $p=7$, the dimension is 8, while at $p=11$, the dimension is 56. 
\end{example}

\begin{remark}
In an alternative approach, one could use obstruction theory on Henn's resolution
\[\E^{hN}\rightarrow X_0\rightarrow X_1\rightarrow\cdots\rightarrow X_n\]
of $\E^{hN}$ (\Cref{thm: Henntopres}) to bound the size of $\kappa_n^N$ by modifying the argument in \cite[4.4.1.(iii)]{heard2014thesis}. Consider the homomorphism $\kappa_n^N\rightarrow\kappa_n^F$ sending $Y$ to $Y\otimes_{\E^{hN}} \E^{hF}$. An upper bound for the kernel is given by the amount of obstructions to lifting a non-exotic $\E^{hF}$-module $Y\otimes_{\E^{hN}} \E^{hF}$ to a non-exotic $\E^{hN}$-module $Y$ via the spectral sequence associated to the resolution of $\E^{hN}$ tensored with $Y$. The obstructions live in $\pi_sX_{s+1}$ where $0\leq s\leq n-1$. Then a straightforward combinatorial argument shows that the size of obstructions is precisely $H^{2n+1}(N,\E_{2n})$, which recovers the upper bound of $\kappa_n^N$ given in \Cref{thm: main corollary}.
\end{remark}

\bibliographystyle{alpha}
\bibliography{bib}

\newcommand{\etalchar}[1]{$^{#1}$}
\begin{thebibliography}{GHMR15}

\bibitem[BB20]{BB}
Tobias Barthel and Agn\`es Beaudry.
\newblock Chromatic structures in stable homotopy theory.
\newblock In {\em Handbook of homotopy theory}, CRC Press/Chapman Hall Handb.
  Math. Ser., pages 163--220. CRC Press, Boca Raton, FL, [2020] \copyright
  2020.

\bibitem[BBG{\etalchar{+}}22]{BBGHPS}
Agnes Beaudry, Irina Bobkova, Paul~G Goerss, Hans-Werner Henn, Viet-Cuong Pham,
  and Vesna Stojanoska.
\newblock The exotic {$K(2)$}-local {P}icard group at the prime 2, 2022.
\newblock Available at https://arxiv.org/abs/2212.07858.

\bibitem[BBGS22]{BBGS}
Tobias Barthel, Agn\`es Beaudry, Paul~G. Goerss, and Vesna Stojanoska.
\newblock Constructing the determinant sphere using a {T}ate twist.
\newblock {\em Math. Z.}, 301(1):255--274, 2022.

\bibitem[BBHS20]{BBHS}
Agn\`es Beaudry, Irina Bobkova, Michael Hill, and Vesna Stojanoska.
\newblock Invertible {$K(2)$}-local {$E$}-modules in {$C_4$}-spectra.
\newblock {\em Algebr. Geom. Topol.}, 20(7):3423--3503, 2020.

\bibitem[BGH22]{beaudry2022chromatic}
Agn\`es Beaudry, Paul~G. Goerss, and Hans-Werner Henn.
\newblock Chromatic splitting for the {$K(2)$}-local sphere at {$p = 2$}.
\newblock {\em Geom. Topol.}, 26(1):377--476, 2022.

\bibitem[BGHS22]{BGHS-duality}
Agn\`es Beaudry, Paul~G. Goerss, Michael~J. Hopkins, and Vesna Stojanoska.
\newblock Dualizing spheres for compact {$p$}-adic analytic groups and duality
  in chromatic homotopy.
\newblock {\em Invent. Math.}, 229(3):1301--1434, 2022.

\bibitem[Boa99]{boardmanConverges}
J.~Michael Boardman.
\newblock Conditionally convergent spectral sequences.
\newblock In {\em Homotopy invariant algebraic structures ({B}altimore, {MD},
  1998)}, volume 239 of {\em Contemp. Math.}, pages 49--84. Amer. Math. Soc.,
  Providence, RI, 1999.

\bibitem[Bro94]{brown}
Kenneth~S. Brown.
\newblock {\em Cohomology of groups}, volume~87 of {\em Graduate Texts in
  Mathematics}.
\newblock Springer-Verlag, New York, 1994.
\newblock Corrected reprint of the 1982 original.

\bibitem[CZ24]{culver}
Dominic~Leon Culver and Ningchuan Zhang.
\newblock Exotic {P}icard groups and chromatic vanishing via the
  {G}ross-{H}opkins duality.
\newblock {\em Topology Appl.}, 341:Paper No. 108742, 32, 2024.

\bibitem[DH95]{devinatz1995}
Ethan~S. Devinatz and Michael~J. Hopkins.
\newblock The action of the {M}orava stabilizer group on the {L}ubin-{T}ate
  moduli space of lifts.
\newblock {\em Amer. J. Math.}, 117(3):669--710, 1995.

\bibitem[DH04]{devinatzhopkins}
Ethan~S. Devinatz and Michael~J. Hopkins.
\newblock Homotopy fixed point spectra for closed subgroups of the {M}orava
  stabilizer groups.
\newblock {\em Topology}, 43(1):1--47, 2004.

\bibitem[GH04]{GoerssHopkins}
P.~G. Goerss and M.~J. Hopkins.
\newblock Moduli spaces of commutative ring spectra.
\newblock In {\em Structured ring spectra}, volume 315 of {\em London Math.
  Soc. Lecture Note Ser.}, pages 151--200. Cambridge Univ. Press, Cambridge,
  2004.

\bibitem[GH22]{goerss2021comparing}
Paul~G. Goerss and Michael~J. Hopkins.
\newblock Comparing dualities in the {$K(n)$}-local category.
\newblock In {\em Equivariant topology and derived algebra}, volume 474 of {\em
  London Math. Soc. Lecture Note Ser.}, pages 1--38. Cambridge Univ. Press,
  Cambridge, 2022.

\bibitem[GHMR05]{GHMR}
P.~Goerss, H.-W. Henn, M.~Mahowald, and C.~Rezk.
\newblock A resolution of the {$K(2)$}-local sphere at the prime 3.
\newblock {\em Ann. of Math. (2)}, 162(2):777--822, 2005.

\bibitem[GHMR15]{ghmrpicard}
Paul Goerss, Hans-Werner Henn, Mark Mahowald, and Charles Rezk.
\newblock On {H}opkins' {P}icard groups for the prime 3 and chromatic level 2.
\newblock {\em J. Topol.}, 8(1):267--294, 2015.

\bibitem[GL21]{GepnerLawson}
David Gepner and Tyler Lawson.
\newblock Brauer groups and {G}alois cohomology of commutative ring spectra.
\newblock {\em Compos. Math.}, 157(6):1211--1264, 2021.

\bibitem[GM95]{GreenleesMay}
J.~P.~C. Greenlees and J.~P. May.
\newblock Generalized {T}ate cohomology.
\newblock {\em Mem. Amer. Math. Soc.}, 113(543):viii+178, 1995.

\bibitem[GR61]{combinatorics}
E.~N. Gilbert and John Riordan.
\newblock Symmetry types of periodic sequences.
\newblock {\em Illinois J. Math.}, 5:657--665, 1961.

\bibitem[Hea14]{heard2014thesis}
Drew Heard.
\newblock {\em Morava modules and the {$K(n)$}-local {P}icard group}.
\newblock PhD thesis, University of {M}elbourne, 2014.

\bibitem[Hea15]{drewTate}
Drew Heard.
\newblock The {T}ate spectrum of the higher real {$K$}-theories at height
  $n=p-1$, 2015.
\newblock Available at https://arxiv.org/abs/1501.07759.

\bibitem[Hen98]{henncentralizers}
Hans-Werner Henn.
\newblock Centralizers of elementary abelian {$p$}-subgroups and mod-{$p$}
  cohomology of profinite groups.
\newblock {\em Duke Math. J.}, 91(3):561--585, 1998.

\bibitem[Hen07]{henn}
Hans-Werner Henn.
\newblock On finite resolutions of {$K(n)$}-local spheres.
\newblock In {\em Elliptic cohomology}, volume 342 of {\em London Math. Soc.
  Lecture Note Ser.}, pages 122--169. Cambridge Univ. Press, Cambridge, 2007.

\bibitem[HG94]{GrossHopkins}
M.~J. Hopkins and B.~H. Gross.
\newblock Equivariant vector bundles on the {L}ubin-{T}ate moduli space.
\newblock In {\em Topology and representation theory ({E}vanston, {IL}, 1992)},
  volume 158 of {\em Contemp. Math.}, pages 23--88. Amer. Math. Soc.,
  Providence, RI, 1994.

\bibitem[HKM13]{HennKaramanovMahowald}
Hans-Werner Henn, Nasko Karamanov, and Mark Mahowald.
\newblock The homotopy of the {$K(2)$}-local {M}oore spectrum at the prime 3
  revisited.
\newblock {\em Math. Z.}, 275(3-4):953--1004, 2013.

\bibitem[HLS21]{HeardLiShi}
Drew Heard, Guchuan Li, and XiaoLin~Danny Shi.
\newblock Picard groups and duality for real {M}orava {$E$}-theories.
\newblock {\em Algebr. Geom. Topol.}, 21(6):2703--2760, 2021.

\bibitem[HMS94]{hopkins-mahowald-sadovsky}
Michael~J. Hopkins, Mark Mahowald, and Hal Sadofsky.
\newblock Constructions of elements in {P}icard groups.
\newblock In {\em Topology and representation theory ({E}vanston, {IL}, 1992)},
  volume 158 of {\em Contemp. Math.}, pages 89--126. Amer. Math. Soc.,
  Providence, RI, 1994.

\bibitem[HMS17]{HMS}
Drew Heard, Akhil Mathew, and Vesna Stojanoska.
\newblock Picard groups of higher real {$K$}-theory spectra at height {$p-1$}.
\newblock {\em Compos. Math.}, 153(9):1820--1854, 2017.

\bibitem[Kar10]{Karamanov}
Nasko Karamanov.
\newblock On {H}opkins' {P}icard group {${\rm Pic}_2$} at the prime 3.
\newblock {\em Algebr. Geom. Topol.}, 10(1):275--292, 2010.

\bibitem[Lad13]{Lader}
Olivier Lader.
\newblock {\em {Une r{\'e}solution projective pour le second groupe de Morava
  pour p $\ge$ 5 et applications}}.
\newblock Theses, {Universit{\'e} de Strasbourg}, October 2013.

\bibitem[LT66]{LubinTate}
Jonathan Lubin and John Tate.
\newblock Formal moduli for one-parameter formal {L}ie groups.
\newblock {\em Bull. Soc. Math. France}, 94:49--59, 1966.

\bibitem[Lur09]{Lurie-Elliptic}
J.~Lurie.
\newblock A survey of elliptic cohomology.
\newblock In {\em Algebraic topology}, volume~4 of {\em Abel Symp.}, pages
  219--277. Springer, Berlin, 2009.

\bibitem[LZ23]{lizhang}
Guchuan Li and Ningchuan Zhang.
\newblock The inverse limit topology and profinite descent on picard groups in
  $k(n)$-local homotopy theory, 2023.
\newblock Available at https://arxiv.org/abs/2309.05039.

\bibitem[Mat16]{mathew2016galois}
Akhil Mathew.
\newblock The {G}alois group of a stable homotopy theory.
\newblock {\em Adv. Math.}, 291:403--541, 2016.

\bibitem[Mil81]{miller}
Haynes~R. Miller.
\newblock On relations between {A}dams spectral sequences, with an application
  to the stable homotopy of a {M}oore space.
\newblock {\em J. Pure Appl. Algebra}, 20(3):287--312, 1981.

\bibitem[Mor23]{mor}
Itamar Mor.
\newblock Picard and {B}rauer groups of $ {K }(n) $-local spectra via profinite
  {G}alois descent, 2023.
\newblock Available at https://arxiv.org/abs/2306.05393.

\bibitem[MR77]{MillerRavenel}
Haynes~R. Miller and Douglas~C. Ravenel.
\newblock Morava stabilizer algebras and the localization of {N}ovikov's
  {$E\sb{2}$}-term.
\newblock {\em Duke Math. J.}, 44(2):433--447, 1977.

\bibitem[MS16]{ms}
Akhil Mathew and Vesna Stojanoska.
\newblock The {P}icard group of topological modular forms via descent theory.
\newblock {\em Geom. Topol.}, 20(6):3133--3217, 2016.

\bibitem[Nav10]{Nave}
Lee~S. Nave.
\newblock The {S}mith-{T}oda complex {$V((p+1)/2)$} does not exist.
\newblock {\em Ann. of Math. (2)}, 171(1):491--509, 2010.

\bibitem[Pst21]{Pstragowski-KnAlg}
Piotr Pstr\k{a}gowski.
\newblock Chromatic homotopy theory is algebraic when {$p>n^2+n+1$}.
\newblock {\em Adv. Math.}, 391:Paper No. 107958, 37, 2021.

\bibitem[Pst22]{Pstragowski}
Piotr Pstr{\k{a}}gowski.
\newblock Chromatic {P}icard groups at large primes.
\newblock {\em Proc. Amer. Math. Soc.}, 150(11):4981--4988, 2022.

\bibitem[Rav78]{Ravenel_Arf}
Douglas~C. Ravenel.
\newblock The non-existence of odd primary {A}rf invariant elements in stable
  homotopy.
\newblock {\em Math. Proc. Cambridge Philos. Soc.}, 83(3):429--443, 1978.

\bibitem[Rav86]{ravenelbook}
Douglas~C. Ravenel.
\newblock {\em Complex cobordism and stable homotopy groups of spheres}, volume
  121 of {\em Pure and Applied Mathematics}.
\newblock Academic Press, Inc., Orlando, FL, 1986.

\bibitem[Sch96]{scheiderer}
Claus Scheiderer.
\newblock Farrell cohomology and {B}rown theorems for profinite groups.
\newblock {\em Manuscripta Math.}, 91(2):247--281, 1996.

\bibitem[Sto12]{stojanoskaDuality}
Vesna Stojanoska.
\newblock Duality for topological modular forms.
\newblock {\em Doc. Math.}, 17:271--311, 2012.

\bibitem[Str92]{Strikland-p-adic}
N.~P. Strickland.
\newblock On the {$p$}-adic interpolation of stable homotopy groups.
\newblock In {\em Adams {M}emorial {S}ymposium on {A}lgebraic {T}opology, 2
  ({M}anchester, 1990)}, volume 176 of {\em London Math. Soc. Lecture Note
  Ser.}, pages 45--54. Cambridge Univ. Press, Cambridge, 1992.

\bibitem[Sym04]{Symonds}
Peter Symonds.
\newblock The {T}ate-{F}arrell cohomology of the {M}orava stabilizer group
  {$S_{p-1}$} with coefficients in {$E_{p-1}$}.
\newblock In {\em Homotopy theory: relations with algebraic geometry, group
  cohomology, and algebraic {$K$}-theory}, volume 346 of {\em Contemp. Math.},
  pages 485--492. Amer. Math. Soc., Providence, RI, 2004.

\bibitem[Sym07]{symondsres}
Peter Symonds.
\newblock Permutation complexes for profinite groups.
\newblock {\em Comment. Math. Helv.}, 82(1):1--37, 2007.

\end{thebibliography}
\end{document}